\documentclass[10pt, a4paper, twoside]{amsart}
\usepackage{amsfonts}
\usepackage{mathrsfs}
\usepackage{amsmath}
\usepackage{amssymb}
\usepackage{fancyhdr}
\usepackage{graphicx}
\usepackage{color}

\numberwithin{equation}{section}

\allowdisplaybreaks

\newcommand\RR{\mathbb{R}}
\newcommand\CC{\mathbb{C}}
\newcommand\NN{\mathbb{N}}

\newcommand\BB{\mathbb{B}}
\newcommand\HH{\mathbb{H}}
\newcommand\Id{\mathrm{Id}}

\newcommand\diag{\mathrm{diag}}

\newcommand\FL{\mathrm{FL}}
\newcommand\FR{\mathrm{FR}}
\newcommand\FF{\mathrm{FF}}
\newcommand\F{\mathrm{F}}
\newcommand\R{\mathrm{R}}

\renewcommand\Re{\operatorname{Re}}
\renewcommand\Im{\operatorname{Im}}

\newcommand\CIdot{\dot C^\infty}

\newtheorem{theorem}{Theorem}
\newtheorem{lemma}[theorem]{Lemma}
\newtheorem{proposition}[theorem]{Proposition}
\newtheorem{corollary}[theorem]{Corollary}

\newtheorem{question}[theorem]{Question}

\theoremstyle{remark}
\newtheorem{remark}[theorem]{Remark}

\begin{document}

\title{\textbf{The Heat Kernel on Asymptotically Hyperbolic Manifolds} }

\author{Xi Chen and Andrew Hassell}

\keywords{Asymptotically hyperbolic manifolds, heat kernel, resolvent, Riesz transform.
}
\date{}

\begin{abstract}
Upper and lower bounds on the heat kernel on complete Riemannian manifolds were obtained in a series of pioneering works due to Cheng-Li-Yau, Cheeger-Yau and Li-Yau. However, these estimates do not give a complete picture of the heat kernel for all times and all pairs of points --- in particular, there is a considerable gap between available upper and lower bounds at large distances and/or large times.
 Inspired by the work of Davies-Mandouvalos on $\mathbb{H}^{n + 1}$, we study heat kernel bounds on Cartan-Hadamard manifolds that are asymptotically hyperbolic in the sense of Mazzeo-Melrose.
 Under the assumption of no eigenvalues and no resonance at the bottom of the continuous spectrum, we show that the heat kernel on such manifolds is comparable to the heat kernel on hyperbolic space of the same dimension (expressed as a function of time $t$ and geodesic distance $r$), \emph{uniformly} for all $t \in (0, \infty)$ and all $r \in [0, \infty)$. Our approach is microlocal and based on the resolvent on asymptotically hyperbolic manifolds, constructed in the celebrated work of Mazzeo-Melrose, as well as its high energy asymptotic, due to Melrose-S\`{a} Barreto-Vasy.
 
 As an application, we show boundedness on $L^p$ of the Riesz transform $\nabla (\Delta - n^2/4 + \lambda^2)^{-1/2}$, for $\lambda \in (0, n/2]$, on such manifolds, for $p$ satisfying $| p^{-1} - 2^{-1}| < \lambda/n$. For $\lambda = n/2$ (the standard Riesz transform $\nabla \Delta^{-1/2}$),  this was previously shown by Lohou\'e in a more general setting.  
 \end{abstract}

\maketitle
\section{Introduction}

The heat kernel $H(t, z, z')$ on a manifold $M$ is the positive fundamental solution of the following Cauchy problem in $(0, + \infty) \times M$
\begin{equation}\label{eqn:heat equation}
\left\{
\begin{array}{c}\frac{\partial u}{\partial t}(t, z) = - \Delta_M u (t, z) \\ u(0, z) = \delta(z - z')
\end{array}
\right.,
\end{equation}
where $t \in (0, \infty)$, $z, z' \in M$ and $\Delta_M$ is the positive Laplacian on $M$. Equivalently,  $H(t, z, z')$ is the Schwartz kernel of the heat semigroup $e^{- t \Delta}$. In particular, the heat kernel on Euclidean space $\mathbb{R}^d$ is given by  \begin{equation}\label{eqn:heat kernel euclidean}H(t, z, z') = \frac{1}{(4 \pi t)^{d/2}} \exp\bigg(- \frac{|z - z'|^2}{4t}\bigg).\end{equation}

The Gaussian decay away from the diagonal, for each fixed $t$, is a typical feature of heat kernels.  Cheng-Li-Yau \cite{Cheng-Li-Yau} proved Gaussian upper bounds for heat kernels on Riemannian manifolds:

\begin{theorem}[Cheng-Li-Yau]Let $M$ be a complete non-compact Riemannian manifold whose sectional curvature is bounded from below and above. For any constant $C > 4$, there exists $C_1$ depending on $C$, $T$, $z \in M$, the bounds of the curvature of $M$ so that for all $t \in [0, T]$ the heat kernel $H(t, z, z')$ obeys \begin{equation}\label{eqn : Cheng-Li-Yau}H(z, z', t) \leq \frac{C_1(C, T, z)}{ |B_{\sqrt{t}}(z)|} \exp\bigg(- \frac{r^2(z, z')}{Ct}\bigg),\end{equation} where $r(z, z')$ is the geodesic distance on $M$.
\end{theorem}
Around the same time, Cheeger-Yau \cite{Cheeger-Yau} gave a lower bound result of the heat kernel.
 \begin{theorem}[Cheeger-Yau]\label{thm : cheeger-yau}Let $M$ be a Riemannian manifold and $B_R(z)$ a compact metric ball in $M$. The heat kernel $H(t, z, z')$ on $B_R(z)$ is bounded from below by $\mathcal{H}(t, z, z')$, where $\mathcal{H}$ is the heat kernel of a model space $\mathcal{M}$ to $M$, where the mean curvatures of $\mathcal{M}$'s distance spheres are greater than the corresponding mean curvatures in $M$.
\end{theorem}
Later Li-Yau \cite{Li-Yau} improved these results, in particular establishing long-time estimates for the heat kernel. \begin{theorem}[Li-Yau]Let $M$ be a manifold with Ricci curvature bounded by $-K$ from below for some $K > 0$. Then the following estimates on $M$ hold. When $K \leq 0$ i.e. the Ricci curvature is always non-negative, there exist positive constants $c, C$ such that 
\begin{equation}\label{eqn : Li-Yau} \frac{c}{|B_{\sqrt{t}}(z)|} \exp\bigg(- C \frac{r^2(z, z')}{t}\bigg) \leq H(t, z, z') \leq \frac{C}{|B_{\sqrt{t}}(z)|} \exp\bigg(- c \frac{r^2(z, z')}{t}\bigg)
\end{equation}
for all $z, z' \in M$ and $t > 0$.  When $K > 0$ there are positive constants $C, c_1, c_2$ such that  $$H(t, z, z') \leq \frac{C}{|B_{\sqrt{t}}(z)|} \exp\bigg(c_1Kt - c_2 \frac{r^2(z, z')}{t}\bigg),$$for all $z, z' \in M$ and $t > 0$.\end{theorem}

In this paper, we will study the heat kernel on asymptotically hyperbolic, Cartan-Hadamard manifolds. On this relatively restricted class of manifolds, we can hope to find tighter upper and lower bounds on the heat kernel, which are valid for long time as well as bounded time, and for both short and large distances.
From this point of view, none of these celebrated results are fully satisfactory. In fact, Cheng-Li-Yau estimates only applies to the short time case, whilst there would be an exponential growth term as $t \to \infty$ in the Li-Yau estimates in the case of negative curvature.

One gets a clue as to what to expect on asymptotically hyperbolic manifolds from the paper of Davies and Mandouvalos \cite{Davies-Mandouvalos}. They proved
 \begin{theorem}[Davies-Mandouvalos]\label{Davies-Mandouvalos estimates}The heat kernel $e^{- t\Delta_{\mathbb{H}^{n + 1}}}$ on $\mathbb{H}^{n + 1}$ is equivalent to \begin{equation}\label{DM} t^{-(n + 1)/2} \exp\big(- \frac{n^2t}{4} - \frac{r^2}{4t} - \frac{nr}{2} \big) \cdot (1 + r + t)^{n/2 - 1}(1 + r),\end{equation} uniformly for $0 \leq r < \infty$ and $0 < t < \infty$, where $r = d(z, z')$ is the geodesic distance on $\mathbb{H}^{n + 1}$.
 \end{theorem}
 We call the right hand side of \eqref{DM} the Davies-Mandouvalos quantity. From it, we can understand certain features of the heat kernel that we could expect to hold more generally on asymptotically hyperbolic spaces. The exponential decay in time is clearly related to the bottom of the spectrum, which is $n^2/4$ on $\HH^{n+1}$ --- this is clear by considering the expression of the heat semigroup in terms of the spectral measure, as in \eqref{eqn : heat kernel expression sm}. Also, the exponential decay in space, $e^{-nr/2}$, independent of time, is the reciprocal of the square root of the volume growth which is also present in the resolvent kernel (\cite{Mazzeo-Melrose}, or see \eqref{Rod large r}).

Inspired by Theorem~\ref{Davies-Mandouvalos estimates}, we ask if the equivalence of the heat kernel with the Davies-Mandouvalos quantity still holds on \emph{asymptotically} hyperbolic manifolds. More specifically, an $n + 1$-dimensional asymptotically hyperbolic manifold $X$ is the interior $X^\circ$ of a compact manifold $X$ with boundary $\partial X$ and endowed with an asymptotically hyperbolic metric. To define the metric we denote $x$ a boundary defining function for $X$. A metric $g$ is said to be conformally compact, if $x^2 g$ is a Riemannian metric and extends smoothly to the closure of $X$. The interior $X^\circ$ of $X$ is thus metrically complete, which amounts to that the boundary is at spatial infinity. Mazzeo \cite{Mazzeo-JDG-1988} showed its sectional curvature approaches $- |dx|^2_{x^2 g}$ as $x \rightarrow 0$; i.e. at `infinity'. A conformally compact metric $g$ is said to be asymptotically hyperbolic if $- |dx|^2_{x^2 g} = - 1$ at boundary, that is, if the sectional curvatures approach $-1$ as $x \to 0$. Furthermore, an asymptotically hyperbolic Einstein manifold $(X, g)$ is an asymptotically hyperbolic manifold with $$\text{Ric}\, g = - n g.$$
A basic model of asymptotically hyperbolic manifolds is the well-known Poincar\'{e} disc, which is the ball $\mathbb{B}^{n + 1} = \{z \in \mathbb{R}^{n + 1} : |z| < 1\}$ equipped with metric
\begin{equation}
\frac{4 dz^2}{(1 - |z|^2)^2}.
\label{Poincare-disc}\end{equation}
In a collar neighbourhood of the boundary, one can write
\begin{equation}
g = \frac{dx^2}{x^2} + \frac{g_0(x, y, dy)}{x^2},
\label{g-normalform}\end{equation}
 where $x$ is a boundary defining function, and $g_0$ is a metric on the boundary but depending parametrically on $x$. For example, the metric \eqref{Poincare-disc} can be written in this form: we take as boundary defining function $\rho = (1 - |z|)(1+|z|)^{-1}$. Let $\theta$ be coordinates on $S^n$, and write the standard metric on the sphere as $d\theta^2$. Then the Poincar\'e metric takes the form
$$(d\rho^2 + (1 - \rho)^2(1+\rho)^{-2} d\theta^2) / \rho^2$$ near $\rho = 0,$ which is of the form \eqref{g-normalform}.

Such manifolds are of great importance in the AdS-CFT correspondence, general relativity, conformal geometry and scattering theory. Let us mention that the Anti-de Sitter space is an example of the Lorentzian version of asymptotically hyperbolic manifolds; the static part of the d'Alembertian on the de Sitter-Schwarzschild black hole model is a $0$-differential operator (of the same kind with the Laplacian on asymptotically hyperbolic manifolds); and the scattering matrix on asymptotically hyperbolic manifolds is conformally invariant.


Our main results concern asymptotically hyperbolic manifolds $X^\circ$ that are also Cartan-Hada\-mard manifolds, that is, they are simply connected with sectional curvatures everywhere negative (recall that a general asymptotically hyperbolic manifold must have negative curvature near infinity, but might have positive curvature on a compact set). By virtue of being Cartan-Hadamard, $X^\circ$ is diffeomorphic to $\RR^n$, with the exponential map based at any point furnishing a global diffeomorphism. It follows that the boundary $\partial X$ of the compactification, $X$, of $X^\circ$ is a sphere.

We impose extra spectral assumptions. It follows from Mazzeo \cite{Mazzeo-1991} that the spectrum of the Laplacian, $\Delta_X$, on\footnote{We will denote the Laplacian on $X^\circ$ by $\Delta_X$, even though $\Delta_{X^\circ}$ would be more accurate} $X^\circ$ consists of an at most finite number of eigenvalues in the interval $(0, n^2/4)$, each with finite multiplicity, and continuous spectrum on the interval $[n^2/4, \infty)$ with no embedded eigenvalues.
Thus the resolvent $(\Delta_X - n^2 - \lambda^2)^{-1}$ is holomorphic for  $\Im \lambda \leq 0$ except for a finite number of poles at $\lambda = -i\mu_j$, $\mu_j > 0$, whenever $n^2 - \mu_j^2$ is an eigenvalue of $\Delta_X$, and possibly at $\lambda = 0$,  corresponding to the bottom of the continuous spectrum.  The classic work of Mazzeo and Melrose \cite{Mazzeo-Melrose} shows that this resolvent meromorphically continues to  a neighbourhood of $\Im \lambda \geq 0$. (We say more about this below.)
\emph{We will assume that $\Delta_X$ has no eigenvalues and $(\Delta_X - n^2 - \lambda^2)^{-1}$ is holomorphic in a neighbourhood of $\lambda = 0$.} We phrase this assumption as ``no eigenvalues and no resonance at the bottom of the spectrum''. Under this assumption, $\Delta_X$ has absolutely continuous spectrum.

While this is definitely a restriction, there are plenty of interesting examples. A particularly noteworthy class of examples is furnished by asymptotically hyperbolic Einstein manifolds. Graham and Lee \cite{Graham-Lee} proved the existence of asymptotically hyperbolic Einstein metrics for conformal structures at infinity that are sufficiently close in $C^{k, \alpha}$ norm to the standard conformal structure on the sphere (where $k=3$ if $n=3$ and $k=2$ for $n \geq 4$). Then, Lee \cite{Lee-CAG-1995} showed the absence of $L^2$-eigenvalues on such manifolds,  following work by Schoen-Yau \cite{Schoen-Yau-Invent1988} and Sullivan \cite{Sullivan-jdg1987}. Next, Guillarmou-Qing \cite{Guillarmou-Qing} showed that on asymptotically hyperbolic Einstein manifolds with conformal infinity of positive Yamabe type, there is no resonance at the bottom of the spectrum. We also refer the interested readers to Bouclet \cite{Bouclet} for more examples.

\begin{theorem}[Heat kernel on asymptotically hyperbolic manifolds]\label{thm : heat kernel bounds}Let $(X,g)$ be an $(n + 1)$-dimensional asymptotically hyperbolic Cartan-Hadamard manifold with no eigenvalues and no resonance at the bottom of the spectrum.  Let $H(t,z,z')$ be the heat kernel on $(X, g)$.

Then $H(t,z,z')$ is equivalent to the Davies-Mandouvalos quantity, i.e. bounded above and below by multiples of \eqref{DM}, uniformly over all times $t \in (0, \infty)$ and distances $r = d(z,z') \in (0, \infty)$.
\end{theorem}

\begin{remark}\label{rem:open} This result motivates posing similar questions in the Euclidean setting: are there asymptotically Euclidean metrics on $\RR^d$, not diffeomorphic to the standard flat metric, for which the heat kernel is bounded above and below by multiples of the Euclidean heat kernel \eqref{eqn:heat kernel euclidean}, uniformly for all distances and all times? We discuss this more in the final section of this paper. \end{remark}

\begin{remark} In particular, this theorem shows that on asymptotically hyperbolic Cartan-Hadamard manifolds, the  Gaussian decay in space for fixed time, $\exp(-r^2/4t)$, occurs with the sharp constant $4$ in the denominator, just as in Euclidean or hyperbolic space. This improves upon the results obtained by Cheng-Li-Yau and Li-Yau in more general geometric settings. The sharp constant is somewhat more significant in hyperbolic geometry compared to Euclidean. In fact, in an asymptotically Euclidean geometry, Gaussian decay ``with the wrong constant'', that is, $t^{-d/2} \exp(-r^2/Ct)$ for some $C > 4$, is still an $L^1$ function of, say, the left space variable (holding time and the right variable fixed), and the $L^1$ norm is uniformly bounded in time, regardless of the value of $C$. By contrast, in hyperbolic settings, while the Gaussian decay ensures the kernel is $L^1$ for each fixed time, a constant larger than $4$ will mean that the $L^1$ norm grows exponentially in time. 
\end{remark}

\begin{remark} Let us comment on the necessity of the assumptions in Theorem~\ref{thm : heat kernel bounds}. The assumption of no eigenvalues is clearly necessary. In fact, if $u(x)$ is an eigenfunction with eigenvalue $n^2/4 - \mu^2$, $\mu > 0$, then $e^{-t\Delta_X} u = e^{-t(n^2/4 - \mu^2)} u$, so the heat kernel cannot decay faster than $e^{-t(n^2/4 - \mu^2)}$. If there is a resonance at the bottom of the spectrum this can also be expected to lead to slower decay of the heat kernel --- this phenomenon is familiar in the case of the Schr\"odinger operators on $\RR^d$; see for example \cite{JK}. Finally, when conjugate points are present then the heat kernel can be expected to be larger than the Euclidean heat kernel for small times, as shown in \cite{Molchanov} and \cite{KS} at least in some special cases. The Cartan-Hadamard assumption therefore could perhaps be weakened to an assumption of no conjugate points, but cannot be dropped altogether. Thus, in the class of asymptotically hyperbolic manifolds, the assumptions seem to be close to optimal. 
\end{remark}

If there is discrete spectrum below the continuous spectrum, we can still give an accurate upper bound on the heat kernel using our method, but it will no longer be equivalent to the Davies-Mandouvalos quantity. Lower bounds seem more difficult to achieve, however. We do not pursue this direction further in the present paper. Instead, we consider an application to the Riesz transform on asymptotically hyperbolic Cartan-Hadamard manifolds. 

The Riesz transform on a complete Riemannian manifold is the operator $\nabla \Delta^{-1/2}$, well-defined as a bounded linear operator  from $L^2(M)$ to $L^2(M; TM)$. One can ask whether it extends from $L^2(M) \cap L^p(M)$ to a bounded linear operator from $L^p(M)$ to $L^p(M; TM)$. If so, one has
$$
\| \nabla f \|_{L^p(M)} \leq C \| \Delta^{1/2} f \|_{L^p(M)},
$$
that is, the operator $\Delta^{1/2}$ controls the full gradient in $L^p$. In 1985, Lohou\'e \cite{Lohoue} showed the boundedness of the Riesz transform on all $L^p$ spaces, $1 < p < \infty$, on Cartan-Hadamard manifolds admitting a spectral gap and with second derivatives bounds on the curvature. 
Coulhon and Duong \cite[Theorem 1.3]{Coulhon-Duong} showed that the Riesz transform on a more general class of spaces with exponential volume growth is bounded on $L^p$ for $p \in (1, 2]$. They only needed to assume that the Laplacian possesses a spectral gap, and that the heat kernel for small time satisfies standard on-diagonal bounds (in addition to the exponential growth condition). Building on this, together with Auscher and Hofmann \cite[Theorem 1.9]{Auscher-Coulhon-Duong-Hofmann}, they showed that the Riesz transform on such spaces is bounded on $L^p$ for all $p$ provided that certain off-diagonal gradient estimates for the heat kernel also hold --- see Theorems~\ref{thm : ACDH} and \ref{thm:ACDH2}. 

For Laplacians with a spectral gap (here equal to $n^2/4$), it is natural also to consider functions of the operator $(\Delta_X - n^2/4)$ --- see \cite{Taylor}, \cite{Chen-Hassell2}. So we could consider the Riesz transform $\nabla (\Delta_X - n^2/4)^{-1/2}$, or more generally, the family $\nabla (\Delta_X - n^2 + \lambda^2)^{-1/2}$, $\lambda \in [0, n/2]$ interpolating between the two. 

In Section~\ref{sec:Riesz} of this paper we establish such gradient estimates, and hence prove 
\begin{theorem}\label{thm:Riesz}
Let $X$ be an asymptotically hyperbolic Cartan-Hadamard manifold with no eigenvalues or  resonance at the bottom of the spectrum.
Then the Riesz transform $T = \nabla (\Delta_X - n^2/4 + \lambda^2)^{-1/2}$ is bounded from $L^p(X)$ to $L^p(X; TX)$ for $\lambda \in (0, n/2]$ and all $p$ satisfying
\begin{equation}
\Big| \frac1{p} - \frac1{2} \Big| < \frac{\lambda}{n}. 
\label{p-lambda-cond}\end{equation}
\end{theorem}

Aside from the pioneering work mentioned above on heat kernel estimates via Riemannian geometry, there are also some literature on the heat kernel on singular spaces from the viewpoint of microlocal analysis. To give the heat equation proof for Atiyah-Singer index theorem and Atiyah-Patodi-Singer index theorem on manifolds with boundary, Melrose \cite{Melrose-APS} introduced the heat calculus and the $b$-heat calculus to construct the resolvent of the heat operator at short times.
Albin \cite{Albin-AdvMath-2007} generalized the heat calculus to edge metrics to prove the renormalized index theorem. The geometric setting is rather broad and covers asymptotically hyperbolic manifolds as well as asymptotically conic manifolds; however, this argument only applies to short times. Sher \cite{Sher-APDE-2013} studied the long time behaviour of the heat kernel on asymptotically conic manifolds and obtained the asymptotic of the renormalized heat trace. Employing the resolvent results on such manifolds, he applied the pull-back and push-forward theorems between appropriate resolvent spaces and heat spaces. However, he does not show the Gaussian decay away from the diagonal, but only exhibits the vanishing order on each face of the heat space. Therefore neither the heat calculus nor the push-forward and pull-back theorems would give an upper bound for the heat kernel on asymptotically hyperbolic manifolds analogous to the Davies-Mandouvalos quantity.

On the other hand, we express the heat kernel via the spectral measure or the resolvent as in \eqref{eqn : heat kernel expression sm} and \eqref{eqn : heat kernel expression res}.  Mazzeo-Melrose \cite{Mazzeo-Melrose} introduced the $0$-calculus to to construct the resolvent for finite spectral parameters and proved that the resolvent extends meromorphically through the continuous spectrum\footnote{In terms of our parametrization of the spectrum, they showed a meromorphic continuation except at $\lambda = im/2$ where $m = 1, 2, \dots$; the resolvent may have essential singularities at these points unless the metric is even at $x=0$, as shown by Guillarmou \cite{Guillarmou}.}. Melrose-S\'{a} Barreto-Vasy \cite{Melrose-Sa Barreto-Vasy} further developed the semiclassical $0$-calculus and determined very precisely the behaviour of the resolvent as $|\lambda| \to \infty$.  To obtain the upper bound in Theorem~\ref{thm : heat kernel bounds}, we shift the contour of integration according to the  method of steepest descent, and use the results of Melrose-S\'{a} Barreto-Vasy, slightly developed in Appendix~\ref{sec:resolvent from parametrix}. The lower bound requires corresponding lower bounds on the resolvent, which we give in Section~\ref{sec:pos}.

The outline of the paper is as follows. In Section \ref{sec:0calculus}, we review the $0$-calculus and major results on asymptotically hyperbolic manifolds. Then Section \ref{sec:H3} is devoted to the model spaces $\mathbb{H}^3$ and $\mathbb{H}^5$, on which we derive the heat kernel via the resolvent. This calculation shows how the indices of the resolvent are related to the heat kernel, and motivates the proof of the upper and lower bounds of the heat kernel in Theorem \ref{thm : heat kernel bounds}. Next, the upper bound in Theorem~\ref{thm : heat kernel bounds} is proved in Section \ref{sec:upper bound}. In Section~\ref{sec:pos},  we establish some positivity of the resolvent kernel for $\lambda$ on the negative imaginary axis, which is the key ingredient for obtaining the lower bound in Theorem~\ref{thm : heat kernel bounds}, proved in Section~\ref{sec:lower bound}. In Section~\ref{sec:gradient}, we prove gradient bounds on the heat kernel, under the same spectral assumptions, and then, in Section~\ref{sec:Riesz}, use these gradient bounds together with work of Auscher-Coulhon-Duong-Hofmann \cite{Auscher-Coulhon-Duong-Hofmann} to establish boundedness of the Riesz transform. In section 9 we discuss some open problems, elaborating on Remark~\ref{rem:open}. 
In the appendix, we prove some results on the resolvent kernel that follow readily from the parametrix of Melrose-S\'{a} Barreto-Vasy. 

 The first author is supported by the general financial grant (Grant No. 2016M591591) from the China Postdoctoral Science Foundation, whilst the second author is supported by Discovery Grants DP150102419 and DP160100941 from the Australian Research Council. The authors would like to thank Pierre Portal, Colin Guillarmou, Hong-Quan Li, Andr\'{a}s Vasy, Xuan Thinh Duong and Michael Cowling for various illuminating conversations. The first author is also grateful to Jun Li and Jiaxing Hong for their continuous encouragement and support.


\section{Analysis on asymptotically hyperbolic manifolds}\label{sec:0calculus}

Let $(X^\circ, g)$ be an $(n + 1)$-dimensional Riemannian asymptotically hyperbolic manifold. Joshi-S\'{a} Barreto\cite{Joshi-Sa Barreto-Acta-2000} show that we have the model form of the metric $g$: near each boundary point, there are local coordinates $(x, y)$, where $x$ is a boundary defining function and $y$ restrict to local coordinates on $\partial X$, such that $g$ takes the form
\begin{equation}
g = \frac{d x^2 + g_0(x, y, dy)}{x^2}.
\label{metric}\end{equation}
where $g_0(x, y, dy)$ is a family of metrics on $\partial X$, smoothly parametrized by $x$.

Consider the Laplacian $\Delta_X$ of metric \eqref{metric}, on $(n+1)$-dimensional asymptotically hyperbolic space $(X, g)$. It is of the form $$- (x\partial_x)^2 + n(x \partial_x) + x^2 Q(x),$$ where $Q(x)$ is a family of second order elliptic operators on $\partial X$ with $Q(0) = \Delta_{\partial X}$ the Laplacian of the boundary $\partial X$. In particular, the Laplacian on the Poincar\'{e} disc reads
$$- (x\partial_x)^2 + n(x \partial_x) + a(x) x^2 \Delta_{S^n},$$
where $a(0) = 1$. The Laplacian $\Delta_X$ is an example of a 0-differential operator, which by definition is a differential operator that can be expressed as a sum of products of smooth vector fields on $X$ each of which vanishes at the boundary. That is, near the boundary, it is generated over $C^\infty(X)$  by $x \partial_x$ and $x \partial_{y_i}$ in local coordinates. As a positive definite combination of such vector fields, the Laplacian $\Delta_X$ is an elliptic 0-differential operator (even though it is obviously not elliptic in the usual sense at the boundary). 

The continuous spectrum of $\Delta_X$ is contained in $[n^2/4 , \infty)$, whilst Mazzeo \cite{Mazzeo-1991} proved the point spectrum, if present, is wholly contained in $(0, n^2/4)$.



The celebrated work of Mazzeo-Melrose \cite{Mazzeo-Melrose} introduced the $0$-calculus and constructed the resolvent for the Laplacian $\Delta_X$. Denote by $\partial (\text{diag} X^2)$ the boundary of the diagonal of $X^2$. On this submanifold, one has following local coordinates near the boundary $$\{(x, y, x, y) : x \in [0, 1), y \in \mathbb{R}^n\} .$$ To introduce the $0$-calculus, Mazzeo-Melrose replaced $X^2$ with the `` 0-double space'' $X^2_0$, which is $X^2$ blown up at the boundary of the diagonal. In the notation of \cite{Melrose-APS}, 
$$
X^2_0 = [X^2; \partial(\text{diag\,} X^2)].
$$
We denote by $\beta_0$ the blow-down map $X^2_0 \to X^2$. More precisely, we replace $\partial (\text{diag} X^2)$ by $S_{++} N(\partial (\text{diag\,} X^2))$ the (closed) doubly inward-pointing part of the spherical normal bundle of $\partial (\text{diag} X^2)$.\footnote{We suggest the reader consult the papers by Mazzeo-Melrose \cite{Mazzeo-Melrose} or by Mazzeo \cite{Mazzeo-JDG-1988} for details of the blow-up.} See Figure \ref{fig: F_0}. Via such $0$-blow-up, one obtains three boundary hypersurfaces \begin{eqnarray*}&\mbox{Front Face:} & \FF = \beta_0^\ast\{(0, y, 0, y) : y \in \mathbb{R}^n\} = S_{++} N(\partial (\text{diag} X^2));\\ &\mbox{Left Face}: & \FL = \beta_0^\ast\{(0, y, x', y') : x' \in [0,\infty), y,y' \in \mathbb{R}^n\};\\ &\mbox{Right Face}: & \FR = \beta_0^\ast\{(x, y, 0, y') :  x' \in [0,\infty), y,y' \in \mathbb{R}^n\}.\end{eqnarray*} 
We denote boundary defining functions for these boundary hypersurfaces by $\rho_F$, $\rho_L$, and $\rho_R$, respectively. 
\begin{center}\begin{figure}
\includegraphics[width=0.5\textwidth]{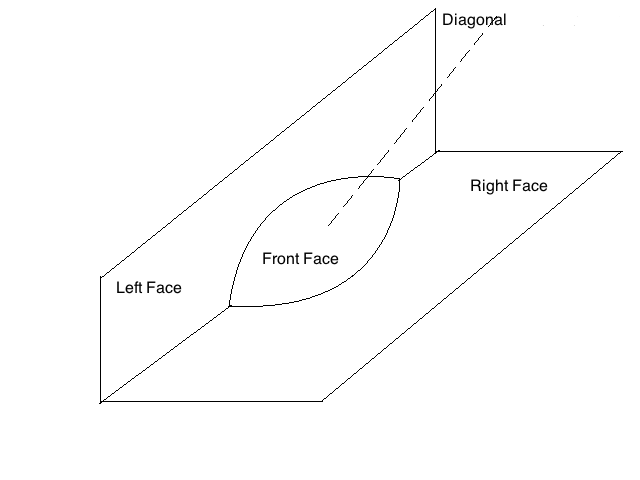}\caption{\label{fig: F_0}The $0$-blown-up double space $X \times_0 X$}\end{figure} \end{center}
On the 0-double space, the 0-differential operators of order $k$ can be characterized as those operators with Schwartz kernel supported on and conormal of order $k$, in the sense of H\"ormander, to the 0-diagonal. It is then natural to define 0-pseudodifferential operators of order $m \in \RR$ as those operators with Schwartz kernel conormal of order $m$ to the 0-diagonal, and vanishing to infinite order at the left and right boundaries. The 0-pseudodifferential operators form an order-filtered algebra, with the property that operators of order 0 are bounded on $L^2(X^\circ, g)$. 

One of the virtues of the space $X^2_0$ is that  the diagonal $\{ r = 0 \}$ is separated from the left and right boundaries where $r \to \infty$. Consequently,  the behaviour of the resolvent kernel on the diagonal is isolated from the behaviour for large $r$. This makes it convenient to use when  constructing parametrices for the resolvent.  In addition, the new face $\FF$ created by blow-up allows one to directly and concretely relate  the Laplacian $\Delta_X$ on asymptotically hyperbolic manifolds to $\Delta_{\mathbb{H}^{n + 1}}$, the Laplacian on hyperbolic space. In fact, one has

\begin{proposition}[\cite{Mazzeo-Melrose}]\label{prop : normal = hyperbolic}Consider $X_0^2$ defined as above. The restriction of $\Delta_X$ to each fibre of the front face $\FF$, defined by freezing the coefficients at boundary points, is the Laplacian $\Delta_{\mathbb{H}^{n + 1}}$ on hyperbolic space.\end{proposition}

Melrose-S\`{a} Barreto-Vasy \cite{Melrose-Sa Barreto-Vasy}  proved a very important asymptotic for the geodesic distance on $X_0^2$ when $X$ is negatively curved.
 \begin{proposition}[\cite{Melrose-Sa Barreto-Vasy}]\label{prop:dist} For any $(z, z')\in X^2$, the geodesic distance function reads
\begin{equation}\label{eqn : geodesic distance} d(z, z') = - \log \rho_L - \log \rho_R + b(z, z'),\footnote{The proof in \cite{Melrose-Sa Barreto-Vasy} is only claimed for metrics close to the hyperbolic metric. However, it applies verbatim to any asymptotically hyperbolic Cartan-Hadamard manifold.}\end{equation} where $b(z, z')$ is $C^\infty$ on $X^2_0 \setminus \diag$. \end{proposition}

Because of this, exponential decay of the kernel as $r \to \infty$ is equivalent to power decay of the kernel at the left and right boundaries $\FL, \FR$ of $X^2_0$.

Consider the resolvent $$ R(\lambda) = (\Delta -n^2/4 - \lambda^2)^{-1}$$ with spectral parameter\footnote{This is related to the $\zeta$ parameter of Mazzeo-Melrose by $\lambda = -i(\zeta - n/2)$} $\lambda$. Mazzeo-Melrose constructed the resolvent on $X^2_0$, and showed that is is the sum of a $0$-pseudodifferential operator of order $-2$ plus a function conormal to the boundary. This result alone justifies the use of the 0-double space $X^2_0$ as the natural space for global analysis on such manifolds. 
Before stating their result, we recall the definition of the $L^\infty$-based conormal functions from \cite{Melrose-conormal}.  Let $Y$ be a manifold with corners. Then we let $\mathcal{A}^0(Y)$ (denoted $S^0(Y)$ in \cite{Melrose-conormal}) denote the set of $L^\infty$ functions $u$ on $Y$ which are smooth in the interior of $Y$, and such that
\begin{multline*} V_1 V_2 \dots V_k u \in L^\infty(Y) \text{ for all smooth vector fields  $V_i$ tangent to each } \\ \text{ boundary hypersurface of $Y$, and all $k = 1, 2, \dots $}
\end{multline*}

\begin{theorem} [\cite{Mazzeo-Melrose, Guillarmou}]\label{thm : Mazzeo-Melrose resolvent} The resolvent $R(\lambda)$, defined for $\{ \Im \lambda < 0 \}$, $\lambda \notin -i(0, n/2)$,  extends to a meromorphic family on $\CC \setminus \frac{i}{2} \NN$ with poles having finite rank residues.  Moreover, the kernel of $\beta^\ast R(\lambda)$ can be decomposed as\footnote{Here we regard these kernels as functions on $X^2_0$ rather than half-densities on $X^2_0$ as in \cite{Mazzeo-Melrose}. To regard as a half-density we simply multiply by the Riemannian half-density on each factor of $X$.}
$$R_{\text{diag}}(\lambda) + R_{od}(\lambda).$$ Here
\begin{itemize}
\item
$R_{\text{diag}}(\lambda)$ is a 0-pseudodifferential operator of order -2,
and
\item
$R_{od}(\lambda)$ is such that  $(\rho_L \rho_R)^{-(n/2 + i\lambda)} R_{od}(\lambda)$ is a meromorphic function of $\lambda$ with values in $\mathcal{A}^0(X^2_0)$, the $L^\infty$-based conormal functions on $X^2_0$. Moreover,
\begin{equation}
(\rho_L \rho_R)^{-(n/2 + i\lambda)}  R_{od}(\lambda) \text{ is continuous on $X^2_0$ up to the boundary,}
\label{MMregularity}\end{equation}
 and therefore has a well-defined restriction to each boundary hypersurface of  $X^2_0$.
 \item
 The resolvent kernel $R(\lambda)$ restricts to each fibre of the front face to be\footnote{See \cite[Section 4, particularly (4.12)]{Mazzeo-Melrose} for the precise sense in which this is true.} the hyperbolic resolvent, $R_{\HH^{n+1}}(\lambda)$ (cf. Proposition~\ref{prop : normal = hyperbolic}).
\end{itemize}
\end{theorem}

\begin{remark} More precise statements can be made about $R_{od}(\lambda)$; it is, in fact, polyhomogeneous conormal in the sense of \cite{Melrose-conormal} with index sets that can be precisely specified. Since we do not need that level of precision in the present paper, we do not give further details.
\end{remark}

This theorem readily implies, \footnote{This was not explicitly addressed in \cite{Mazzeo-Melrose}. See for example the paper of Patterson-Perry \cite{Patterson-Perry}.} 
\begin{corollary}\label{coro : eigenvalues} The resolvent $(\Delta_X - n^2/4 - \lambda^2)^{-1}$, for $\Im \lambda < 0$, has a simple pole at $\lambda = -i\sigma_j$ for each eigenvalue $n^2/4 - \sigma_j^2$ of $\Delta_X$. The residue at $\lambda = -i\sigma_j$ is $(2i \sigma_j)^{-1}$ times the orthogonal projection $P_j$ onto the corresponding eigenspace. The eigenfunctions with eigenvalue $n^2 - \sigma_j^2$ lie in the space $x^{n/2 + \sigma_j}  C^\infty(X)$.
\end{corollary}

For simplicity, we assume that $\Delta_X$ has no eigenvalues. 
Our basic strategy for analyzing the heat kernel is to express it in terms of the spectral measure
\begin{equation}\begin{gathered}
e^{-t(\Delta_X)} = e^{-tn^2/4} \, e^{-t(\Delta_X - n^2/4)} = e^{-tn^2/4} \, \int_0^\infty e^{-t\sigma} dE_{(\Delta_X- n^2/4)}(\sigma)  \, d\sigma 
\end{gathered}\label{eqn : heat kernel expression sm}\end{equation}
and then, via Stone's formula, in terms of the resolvent:
\begin{equation}\begin{gathered}
\phantom{e^{-t(\Delta_X)} = e^{-tn^2/4} \, e^{-t(\Delta_X - n^2/4)}} = 
\frac{\imath}{2\pi} e^{-tn^2/4} \, \int_{-\infty}^\infty e^{-t\lambda^2} R(\lambda - \imath 0) \, 2\lambda \, d\lambda, \quad \sigma = \lambda^2.
\end{gathered}\label{eqn : heat kernel expression res}\end{equation}
Which representation is preferred depends on whether we seek near-diagonal or off-diagonal heat kernel estimates. For off-diagonal estimates, the resolvent representation is more useful, as we can shift the contour of integration in the lower half plane to optimize the estimates. For near diagonal estimates, however, the resolvent has the flaw that the kernel diverges on the diagonal, while the heat kernel is smooth there for $t > 0$. So we use the spectral measure, which is also smooth across the diagonal. It follows that we need spectral measure estimates near the diagonal, and off-diagonal resolvent estimates.

The present authors \cite{Chen-Hassell1, Chen-Hassell2} constructed the limiting resolvent kernel on the spectrum for high energies, and hence deduced spectral measure estimates (for low energies it follows directly from Theorem~\ref{thm : Mazzeo-Melrose resolvent}).

\begin{theorem}[\cite{Chen-Hassell2}]Suppose $(X, g)$ is an $n + 1$-dimensional asymptotically hyperbolic Cartan-Hadamard manifold with no resonance at the bottom of the continuous spectrum and denote the operator $\sqrt{(\Delta_X - n^2/4)_+}$ by $P$. For $\lambda < 1$, the Schwartz kernel of the spectral measure $dE_{P}(\lambda)$ satisfies bounds\footnote{In \cite{Chen-Hassell2}, `microlocalized' estimates are proved. However, when the manifold is Cartan-Hadamard, the microlocalizing operators $Q_i$ are not required, which implies \eqref{eqn : upper bound spectral measure}.} 
\begin{equation}\label{eqn : upper bound spectral measure}
\Big|dE_{P}(\lambda) (z,z') \Big| \leq \left\{ \begin{array}{ll} C \lambda^{2} , &  \text{if} \quad \lambda \leq 1; \\
C \lambda^{n}  & \text{if} \quad \lambda \geq 1. \end{array} \right.
\end{equation}
\end{theorem}

For the resolvent, we only need the off-diagonal behaviour, but we need it in the whole physical half-plane $\Im \lambda \leq 0$, and we need to understand the behaviour uniformly as $|\lambda| \to \infty$. This goes beyond the Mazzeo-Melrose result, which only gives uniform behaviour on compact $\lambda$-sets. To obtain uniform results when $|\lambda|$ is large,  we switch to using the semiclassical calculus. We consider the semiclassical operator $h^2\Delta_X - h^2n^2/4 - \sigma^2$ with $h \rightarrow 0$ and $|\sigma| = 1$. Thus we have $\lambda = \sigma/h$. Apart from the conormality at the diagonal, there are singularities arising from the diagonal cosphere bundle and propagating along bicharacteristics. Therefore the resolvent with a large spectral parameter near $\{\Im \lambda^2 = 0\}$ is the \emph{sum} of a pseudodifferential operator, microlocally supported on the diagonal conormal bundle, and a Fourier integral operator, microlocally supported on the bicharacteristic variety. Readers are referred to \cite{Melrose-Sa Barreto-Vasy, Chen-Hassell1, Wang} for a full description of the microlocal structure.

On the other hand, if $\Im \sigma < 0$ and $h \to 0$, then the  (semiclassical) symbol of $\Delta_X - n^2/4 - \lambda^2$ does not vanish, so the operator is elliptic. 
In this case, the resolvent kernel decays exponentially away from the diagonal. To analyze this decay/oscillation,  Melrose-S\'{a} Barreto-Vasy \cite{Melrose-Sa Barreto-Vasy} introduced a compact space  $X_0^2 \times_1 [0, 1)_h$ incorporating the parameter $h$,  defined as $$[X^2_0 \times [0, 1)_h, (\text{diag} \,X^2_0) \times \{0\}_h].$$
That is, start with $X_0^2 \times [0, 1)_h$ and blow up the diagonal on the semiclassical face $\{h = 0\}$ with blow-down map $\beta_1$. Then following boundary faces will be created. 
\begin{eqnarray*}&&\mathcal{S} = \beta^\ast_1 \Big(\text{diag} \,X^2_0 \times \{0\}_h\Big) \\ &&\mathcal{A} = \beta_1^\ast \Big(\{h = 0\}\Big) \setminus \mathcal{S} \\ &&\mathcal{F} = \beta_1^\ast \Big(\FF\times [0, 1)_h\Big)  \\ &&\mathcal{L} = \beta_1^\ast \Big(\FL\times [0, 1)_h\Big) \\&& \mathcal{R} = \beta_1^\ast \Big(\FR\times [0, 1)_h\Big).
\end{eqnarray*} 
See for example Figure \ref{fig: semiclassicalresolventspace}. \begin{center}\begin{figure}
\includegraphics[width=0.5\textwidth]{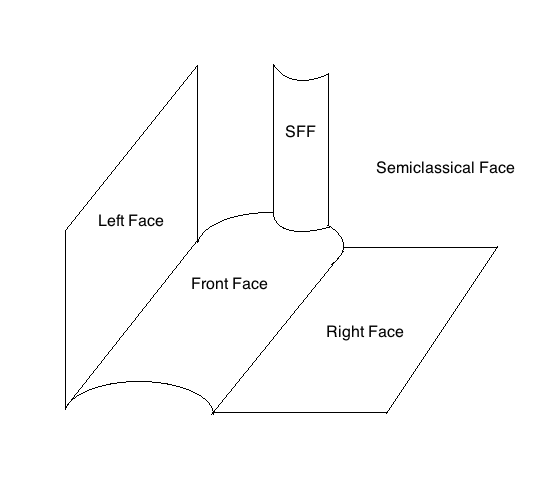}\caption{\label{fig: semiclassicalresolventspace}The semiclassical resolvent space}\end{figure}
\end{center}

Melrose-S\`{a} Barreto-Vasy constructed a parametrix for the semiclassical operator $h^2\Delta_X - h^2n^2/4 - \sigma^2$ living on such spaces --- see Theorem~\ref{thm:sclres}. (Although their result is claimed only for metrics close to the standard metric on hyperbolic space, the result holds for any asymptotically hyperbolic metric that is Cartan-Hadamard.)
Based on this parametrix, we establish in the appendix corresponding properties for the resolvent itself. Writing this in the form $(\Delta_X - n^2/4 - \lambda^2)^{-1}$, $\Im \lambda \leq 0$, we summarize here the main off-diagonal estimates that we need. 

\begin{corollary}\label{coro : resolvent construction} Assume that $(X,g)$ is an asymptotically hyperbolic Cartan-Hadamard manifold with no eigenvalues and no resonance at the bottom of the spectrum. Let $r$ denote geodesic distance on $X \times X$.
Then the resolvent, $R(\lambda) := (\Delta_X - n^2/4 - \lambda^2)^{-1}$ is analytic in a neighbourhood of the closed lower half plane $\Im \lambda \leq 0$, and satisfies in this region of the $\lambda$-plane and for $r (1 + |\lambda|)  \geq 1$ (the `off-digaonal regime')
\begin{equation}\label{eqn : resolvent decomposition}
R(\lambda)(z,z') = e^{-i\lambda r} R_{od}(\lambda)(z,z'), \quad r = d(z,z'),
\end{equation}
where
\begin{itemize}
\item
for $|\lambda| \leq 1$, $R_{od}(\lambda)$ is an element of $(\rho_L \rho_R)^{n/2} \mathcal{A}^0(X^2_0)$, 
\item
for $|\lambda| \geq 1$, $R_{od}(\lambda)$ is of the form
\begin{equation}
\rho_\mathcal{L}^{n/2} \rho_\mathcal{R}^{n/2} \rho_\mathcal{A}^{-n/2 +1} \rho_\mathcal{S}^{-n+1} \mathcal{A}^0\Big(X_0^2 \times_1 [0, 1)_h\Big).
\label{resolvent away from diag}\end{equation}
\end{itemize}
In particular, $R_{od}(\lambda)$ is a kernel bounded pointwise by a multiple of 
\begin{equation}
 (r (1 +|\lambda|))^{n/2 -1} r^{-n+1} = r^{-n/2} (1+|\lambda|)^{n/2-1}
\label{Rod small r}\end{equation}
for $r \leq C$, and, using the relation \eqref{eqn : geodesic distance},
\begin{equation}
e^{-nr/2} (1+|\lambda|)^{n/2-1}
\label{Rod large r}\end{equation}
for $r \geq C$.

\end{corollary}

\begin{remark}Be aware that there is discrepancy of a factor $\lambda^{-2}$ between the semiclassical resolvent and this expression, accounting for the difference of 2 in the powers of boundary defining functions at $|\lambda| = \infty$ between \eqref{eqn : semiclassical resolvent away from diagonal} and \eqref{resolvent away from diag}.\end{remark}


\section{The heat kernel on $\mathbb{H}^3$ and $\mathbb{H}^5$}\label{sec:H3}

In this section, we will calculate the heat kernel on model spaces. The significance of this calculation is that it relates the rate of vanishing of the resolvent at various boundary faces to that of  the heat kernel.

As in $\mathbb{R}^k$, one can calculate the explicit resolvent on $\mathbb{H}^{n + 1}$. In fact, the Schwartz kernel of $(\Delta_{\mathbb{H}^{n + 1}} - n^2/4 - \lambda^2)^{-1}$, for $\Im \lambda < 0$, reads \begin{eqnarray*} &\displaystyle - \frac{1}{2\imath \lambda} \bigg( - \frac{1}{2\pi} \frac{1}{\sinh(r)} \frac{\partial}{\partial r}\bigg)^k e^{- \imath \lambda r} & \mbox{if $n$ is even},\\ &\displaystyle C \int_0^\infty e^{- \imath \lambda w} \Big(\cosh(w) - \cosh(r)\Big)_{+}^{-n/2} \,dw & \mbox{if $n$ is odd},\end{eqnarray*} for $\Im \lambda < 0$, where $r$ is the geodesic distance between two points in the space. We discuss the easiest cases $\mathbb{H}^3$ and $\mathbb{H}^5$ as model spaces.

\subsection*{The heat kernel on $\mathbb{H}^3$}

The limit of the resolvent of $\Delta_{\mathbb{H}^3} - 1$ on the spectral line $\Im \lambda = 0$ is $$\frac{1}{4\pi}\frac{e^{- \imath \lambda r}}{\sinh(r)}.$$ Due to \eqref{eqn : heat kernel expression sm} the heat kernel on $\mathbb{H}^3$ thus takes the form
\begin{equation}
e^{-t\Delta_{\mathbb{H}^3}} = \frac{1}{4\pi^2 \imath} e^{-t} \, \int_{-\infty}^\infty e^{- t \lambda^2- \imath \lambda r} \frac{\lambda\,d\lambda}{\sinh(r)}.
\label{HeatkernelH3}\end{equation}

We will invoke the method of steepest descent\footnote{Details of the method can be found in the book of Erd\'{e}lyi \cite[p.39-40]{Erdelyi}.} to derive the asymptotic of the heat kernel. Consider the phase function $\phi(\lambda)(\lambda) = - t \lambda^2 - \imath r \lambda$ in \eqref{HeatkernelH3}. First,  the saddle point of $\phi$ is $- \imath r/(2t)$ on $\mathbb{C}$. Secondly, we find steepest paths, where $\Im \phi(\lambda) = 0$. To do so, we write $\lambda = a + b\imath$ with $a, b \in \mathbb{R}$ and compute $\Im \phi(a + b\imath) = - 2abt - ar$. This is zero if $b = -r/2t$. So we change the path of integration to $\Im \lambda = -r/2t$. It follows that
$$e^{-t\Delta_{\mathbb{H}^3}} = \frac{1}{4\pi^2 \imath} e^{-t} \, \int_{-\imath r/(2t) - \infty }^{-\imath r/(2t) + \infty} e^{-t\lambda^2 - \imath r\lambda}\frac{\lambda d\lambda}{\sinh(r)}.
$$
Writing $\lambda = -ir/2t + w$, $w \in \RR$, we have $$e^{-t\Delta_{\mathbb{H}^3}} = \frac{1}{4\pi^2 \imath} \, e^{-t} \, \int_{-\infty}^\infty e^{-tw^2 - r^2/(4t)}\frac{w - \imath r/(2t)}{\sinh(r)} \, dw.$$  The odd part of the integrand does not contribute, and we obtain  $$e^{-t\Delta_{\mathbb{H}^3}} = \frac{r  e^{ - r^2/(4t) - t}}{4\pi^2 t \sinh(r)} \int_{0}^{\infty} e^{-t\lambda^2}  d\lambda = \frac{r  e^{ - r^2/(4t) - t}}{(4\pi t)^{3/2} \sinh(r)} .$$
\begin{center}\begin{figure}
\includegraphics[width=0.5\textwidth]{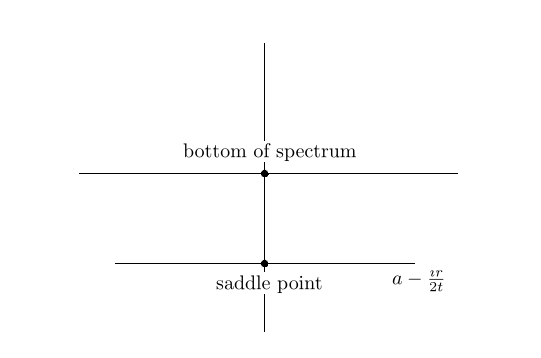}\caption{\label{fig: steepestdescent}Change of integration path}\end{figure}\end{center}

\subsection*{The heat kernel on $\mathbb{H}^5$}

The resolvent $(\Delta_{\mathbb{H}^5} - 4 - \lambda^2)^{-1}$ is $$- \frac{\imath \lambda \sinh(r) + \cosh(r)}{8\pi^2} \frac{e^{- \imath \lambda r}}{(\sinh(r))^3}.$$ Applying the functional calculus and the method of steepest descent as on $\mathbb{H}^3$, we deduce that $$e^{- t \Delta_{\mathbb{H}^5}} = \frac{1}{8 \pi^3 \imath} \, e^{-4t} \, \int_{-\imath r/(2t) - \infty }^{-\imath r/(2t) + \infty}  e^{-t\lambda^2 - \imath r\lambda} \frac{\imath \lambda^2 \sinh(r) + \lambda \cosh(r)}{(\sinh(r))^3}\, d\lambda.$$ A change of variable yields and symmetry yields
\begin{eqnarray*}
\lefteqn{e^{- t \Delta_{\mathbb{H}^5} }} \\ &&= \frac{e^{-r^2/(4t)}}{8 \pi^3 \imath} \, e^{-4t} \, \int_{-\infty}^{-\infty}  e^{-tw^2 } \bigg( \frac{\imath (w^2 - r^2/(4t^2) - \imath w r / t)}{(\sinh(r))^2} + \frac{(w - \imath r /(2t)) \cosh(r)}{ (\sinh(r))^3} \bigg)\, dw\\ &&= \frac{e^{-r^2/(4t)}}{4 \pi^3 } \int_{0}^{\infty}  e^{-t\lambda^2} \bigg( - \frac{ \lambda^2 }{(\sinh(r))^2} + \frac{ r^2 }{4t^2(\sinh(r))^2} + \frac{ r \cosh(r)}{ 2t (\sinh(r))^3}\bigg)\, d\lambda\\ &&= \frac{e^{-r^2/(4t)}}{16 \pi^{5/2} t^{3/2} (\sinh(r))^2}  \bigg( - 1 + \frac{ r^2 }{2t} + \frac{ r \cosh(r)}{ \sinh(r)}\bigg).\end{eqnarray*}

\subsection*{Conclusion}

 Through above calculations on $\mathbb{H}^3$ and $\mathbb{H}^5$, we have come to the following  conclusions to interpret the Davies-Mandouvalos quantity.
\begin{itemize}
\item Through the method of steepest descent, the Gaussian factor, $e^{-r^2/(4t)}$, is obtained from the functional calculus of the heat kernel and the oscillation of the resolvent.
\item The spectral gap, $n^2/4$, results in, $e^{- t n^2 / 4}$, the exponential decay in time.
\item The spatial exponential decay, $e^{- r n / 2}$, of the resolvent transits to the heat kernel.
\item The polynomial order in $r$, which is $n/2$, is determined by the order of the resolvent in $\lambda$, which is $n/2$.
\item Regarding the polynomial factor in time, $t^{-(n + 1)/2}$, $- n/2$ also corresponds to the order of the resolvent in $\lambda$, whilst $-1/2$ is contributed by the integration of $e^{-t\lambda^2}$.
\end{itemize}


\section{Heat kernel upper bounds}\label{sec:upper bound}
 In this section, we shall prove the heat kernel on $X$ is globally (in both time and space) bounded above by the Davies-Mandouvalos quantity \eqref{DM} in Theorem~\ref{thm : heat kernel bounds}.

\begin{proposition}\label{prop : upper}
Let $X$ be an $n+1$-dimensional asymptotically hyperbolic Cartan-Hadamard manifold with  no eigenvalues and no resonance at the bottom of the spectrum.
Then the heat kernel obeys $$ e^{- t \Delta_X}(z, z')  \leq C  t^{-(n + 1)/2} e^{- n^2t/4 - r^2/(4t) - nr/2}  (1 + r + t)^{n/2 - 1}(1 + r),$$ where $r$ is the geodesic distance between $z$ and $z'$.
\end{proposition}

\begin{proof}
We break up the proof into several different regions, depending on the size of $r = d_X(z,z')$ and $t$.
The regions are, where $C, C_1$, etc are any sufficiently large constants, and $\epsilon, \epsilon_1$, etc, are sufficiently small constants.
\begin{itemize}
\item[(i)] $t \geq C, r^2 \leq C$;
\item[(ii)] $t \geq C, \sqrt{C} \leq r \leq Ct$;
\item[(iii)] $t \geq C$, $r \geq Ct$;
\item[(iv)] $t \leq C, r^2 \geq C_1t$;
\item[(v)] $t \leq C_2, r^2 \leq C_3t$.
\end{itemize}
It is easy to check that for arbitrary $C$ (which we will take to be sufficiently large below) and $C_1$ sufficiently large relative to $C$,  these regions cover the entire region $\{ r > 0, t > 0\}$ (see Figure \ref{fig: upperbound}).
\begin{center}\begin{figure}
\includegraphics[width=0.5\textwidth]{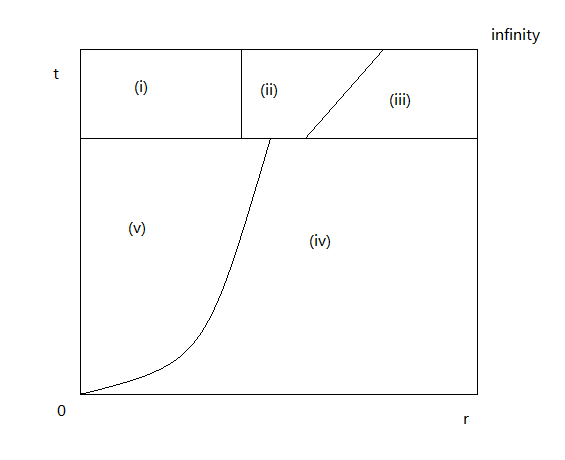}\caption{\label{fig: upperbound}Regions for the proof of the upper bound}\end{figure}
\end{center}

Although we split the estimate into various cases, the basic idea is the same, and it is motivated by the calculation in Section \ref{sec:H3}  for the hyperbolic Laplacian. We express the heat semigroup in terms of either the spectral measure,
\begin{equation}\begin{gathered}
e^{-t(\Delta_X)} = e^{-tn^2/4} \, e^{-t(\Delta_X - n^2/4)} = e^{-tn^2/4} \, \int_0^\infty e^{-t\sigma} dE_{(\Delta_X- n^2/4)}(\sigma) \\ = \frac{\imath}{2\pi} e^{-tn^2/4} \, \int_{0}^\infty e^{-t\lambda^2} dE_{\sqrt{\Delta_X- n^2/4}}(\lambda)
\end{gathered}\label{eqn : heat kernel expression sm2}\end{equation}
or the resolvent, as in \eqref{eqn : heat kernel expression res}:
\begin{equation}\begin{gathered}
e^{-t(\Delta_X - n^2/4)} =  \frac{\imath }{2\pi} \int_{-\infty}^\infty e^{-t\lambda^2} R(\lambda - \imath 0) \, 2\lambda \, d\lambda.
\end{gathered}\label{heatkernelfromres}\end{equation}
The spectral measure representation is sufficient for the near diagonal cases, (i) and (v), where $r^2 \leq Ct$ and so the Gaussian decay factor is comparable to $1$. For the other regions, we exploit the holomorphy of the resolvent $(\Delta_X - n^2/4 - \lambda^2)^{-1}$ in the closed lower half plane. Due to the rapid decrease of the factor $e^{-t\lambda^2}$ as $|\Re \lambda| \to \infty$ for fixed $|\Im \lambda|$, we may shift the contour of integration to $\Im \lambda = c$ for any negative $c$.
Then we use the oscillatory factor $e^{\imath\lambda r}$ from Theorem~\ref{thm:sclres} in combination with the spectral multiplier $e^{-t\lambda^2}$. By shifting the contour to the line $\Im \lambda = -r/2t$, that is writing $\lambda = -\imath r/2t + w$, $w$ real, we find that
\begin{equation}
e^{-i\lambda r - t \lambda^2} = e^{-r^2/4t} e^{-tw^2}.
\label{Gaussian-identity}\end{equation}
The factor $e^{-r^2/4t}$ can be removed from the integral, and we obtain
\begin{equation}\label{w-integral}
e^{n^2t/4} e^{r^2/4t} e^{-t \Delta_X}(z,z') =
\frac{\imath }{2\pi} \int_{-\infty}^\infty e^{-tw^2} R_{od}\big(w - (\frac{\imath r}{2t})\big)(z,z')  \big(w - (\frac{ \imath r}{2t}) \big) \, dw.
\end{equation}

\subsection*{Region (i)} In this region, we need to prove an upper bound of the form $e^{-n^2t/4} t^{-3/2}$.  It is enough to use the spectral measure estimate \eqref{eqn : upper bound spectral measure} in \eqref{eqn : heat kernel expression sm}. This gives us
$$
\int_0^\infty e^{-t\lambda^2} dE_{\sqrt{\Delta_X- n^2/4}}(\lambda)   \leq C \int_0^1 e^{-t\lambda^2} \lambda^2 \, d\lambda
+ C \int_1^\infty e^{-t\lambda^2} \lambda^n \, d\lambda\leq C t^{-3/2},
$$
for $t \geq C > 0$, as required.

\subsection*{ Region (ii)} In this region, we need to show an upper bound of the form $e^{-n^2t/4} e^{-r^2/4t} t^{-3/2} r e^{-nr/2}$. That is, we need to show that the integral on the RHS of \eqref{w-integral} is bounded by a multiple of
$$
r  e^{-nr/2} t^{-3/2}.
$$
In this region, $r$ is large so we can replace the resolvent by $R_{od}(\lambda)$.
We use the fact that $R_{od}(\lambda)$ has an analytic continuation to a neighbourhood of $\lambda = 0$ (note that $R_{od}(\lambda)$ is analytic since the cutoff to large $r$ can be taken independent of $\lambda$), so there is some $\delta > 0$ so that it is an analytic function of $\lambda$ in the closed ball $B(0, \delta)$, satisfying the estimates in \eqref{Rod large r} uniformly there. It follows that the $\lambda$-derivative of $R_{od}(\lambda)$, for $|\lambda| \leq \delta/2$, say, also satisfies the same estimates (up to a factor $2\delta^{-1}$ which we absorb into the constant). Then we compute
\begin{equation}\label{w-integral-region ii}\begin{gathered}
\frac{\imath }{2\pi} \int_{-\infty}^\infty e^{-tw^2} R_{od}\big(w - (\frac{\imath r}{2t}) \big) \big(w - (\frac{\imath r}{2t}) \big) \, dw \\
= \frac{\imath }{2\pi} \int_{-\delta/2}^{\delta/2} e^{-tw^2} R_{od}\big(w - (\frac{\imath r}{2t}) \big) \big(w - (\frac{\imath r}{2t}) \big)   \, dw + \int_{|w| \geq \delta/2} e^{-tw^2} R_{od}\big(w - (\frac{\imath r}{2t}) \big) \big(w - (\frac{\imath r}{2t}) \big)   \, dw \\
=  \frac1{2\pi} \frac{r}{2t} \int_{-\delta/2}^{\delta/2} e^{-tw^2} R_{od}\big(w - (\frac{\imath r}{2t}) \big)    \, dw + \frac{\imath }{2\pi}  \frac1{2t} \int_{-\delta/2}^{\delta/2} e^{-tw^2} (\partial_w R_{od})\big(w - (\frac{\imath r}{2t}) \big)    \, dw \\ + \frac{ \imath e^{-t\delta^2/4}}{2\pi}  \Big( R_{od}(\frac{\delta}{2} - \imath \frac{r}{2t}) - R_{od}(-\frac{\delta}{2} - \imath \frac{r}{2t}) \Big) +  \frac{\imath }{2\pi}  \int_{|w| \geq \delta/2} e^{-tw^2} R_{od}\big(w - (\frac{\imath r}{2t}) \big) \big(w - (\frac{\imath r}{2t}) \big)   \, dw
\end{gathered}
\end{equation}
where we integrated by parts in the second term of the second line. Then we can estimate the sum of these terms, using \eqref{Rod large r}, by
$$
C \Big( \frac{r}{t} e^{-nr/2}  t^{-1/2} + \frac1{t} t^{-1/2} e^{-nr/2} + O(e^{-t\delta^2/4} r e^{-nr/2}) \Big)  \leq C r  e^{-nr/2} t^{-3/2},
$$
as required.

\subsection*{ Region (iii)}  In this region, we need to show that \eqref{w-integral} is bounded by a multiple of
$$
r^{n/2} e^{-nr/2} t^{-(n+1)/2}.
$$
Notice that for $\lambda = -\imath r/2t + w$, $w$ real, then $|\lambda| \geq r/2t \geq 1/2r$. That is, $|\lambda|r \geq 1/2$ along the contour, so we can ignore the pseudodifferential part $R_{diag}(\lambda)$ of the resolvent, and only consider the off-diagonal term. Using \eqref{Rod large r}, we estimate \eqref{w-integral} by
\begin{equation}\label{w-integral-region i}\begin{gathered}
e^{-nr/2} \int_{-\infty}^\infty e^{-tw^2} |w - (\frac{\imath r}{2t})|^{n/2} \, dw \\
\leq C e^{-nr/2} \int_{-\infty}^\infty e^{-tw^2} \Big(  (\frac{r}{2t})^{n/2} + |w|^{n/2} \Big)  \, dw \\
\leq C e^{-nr/2} \Big( r^{n/2} t^{-(n+1)/2} + t^{-n/4 - 1/2} \Big)
\leq C e^{-nr/2} r^{n/2} t^{-(n+1)/2}
\end{gathered}
\end{equation}
where the last line follows because the condition $r \geq Ct$  shows that the term $r^{n/2} t^{-(n+1)/2}$ dominates $t^{-n/4 - 1/2}$.

\subsection*{ Region (iv)}  In this region, given by $t \leq C, r^2 \geq C_1t$, we have $r \geq Ct$ for $C_1$ large enough compared to $C$ (in fact $C_1 \geq C^3$ will do). We need to show that \eqref{w-integral} is bounded by a multiple of
$$
(1+r)^{n/2} e^{-nr/2}t^{-(n+1)/2}
$$
(notice that $r$ may be large or small in this region).
Again, notice that for $\lambda = -\imath r/2t + w$, $w$ real, then $|\lambda|r \geq r^2/2t \geq C/2$. So again we can ignore the pseudodifferential part $R_{diag}(\lambda)$ of the resolvent, and only consider the off-diagonal term. Using \eqref{Rod small r}, we estimate \eqref{w-integral} for small $r$  by
\begin{equation}\label{w-integral-region v}\begin{gathered}
 \int_{-\infty}^\infty e^{-tw^2} r^{-n/2} \Big( 1 + |w - (\frac{\imath r}{2t})|\Big)^{n/2} \, dw \\
\leq C \int_{-\infty}^\infty e^{-tw^2} r^{-n/2} \Big(1 +  (\frac{r}{2t})^{n/2} + |w|^{n/2} \Big)  \, dw \\
\leq C \Big( r^{-n/2} t^{-1/2} + t^{-n/2} t^{-1/2} + r^{-n/2} t^{-n/4 - 1/2} \Big)
\leq C  \Big( t^{-n/4 - 1/2} + t^{-n/2 - 1/2} \Big) \leq C t^{-n/2 - 1/2}
\end{gathered}
\end{equation}
where we used the condition $r \geq (Ct)^{1/2}$ in the second last line, and $t \leq C$ in the last line.

On the other hand, for large $r$, we estimate \eqref{w-integral} by
\begin{equation}\label{w-integral-region vv}\begin{gathered}
e^{-nr/2} \int_{-\infty}^\infty e^{-tw^2} \Big( 1 + |w - (\frac{\imath r}{2t})|\Big)^{n/2} \, dw \\
\leq C e^{-nr/2} \int_{-\infty}^\infty e^{-tw^2}\Big(1 +  (\frac{r}{2t})^{n/2} + |w|^{n/2} \Big)  \, dw \\
\leq C e^{-nr/2} \Big( t^{-1/2} + r^{n/2} t^{-n/2} t^{-1/2} + t^{-n/4 - 1/2} \Big)
 \leq Ce^{-nr/2} r^{n/2} t^{-n/2 - 1/2}.
\end{gathered}
\end{equation}

\subsection*{ Region (v)} In this region, the Gaussian term $e^{-r^2/4t}$ is bounded below, and it is enough to obtain a $C t^{-(n+1)/2}$ upper bound on the heat kernel. This is obtained from the spectral measure estimates \eqref{eqn : upper bound spectral measure} as for region (i). Alternatively, it is implied by Cheng-Li-Yau's upper bound \eqref{eqn : Cheng-Li-Yau}.

\end{proof}


\section{Positivity properties of the resolvent $R(-\imath \lambda)$ on the negative imaginary axis}\label{sec:pos}
To obtain lower bounds, we need some strict positivity properties of the resolvent kernel $R(-\imath \lambda)$ on the negative imaginary axis, that is, for $\lambda > 0$. We will also need to identify leading terms in certain asymptotic expansions. For this purpose we will use the following elementary lemma.

\begin{lemma}\label{lem:fexp}
Suppose that the function $u(w)$, $w \in \RR$, is smooth, with each derivative growing at most polynomially at infinity.
Then the function
\begin{equation}
f(t) := \int_{-\infty}^\infty e^{-tw^2} u(w) \, dw
\label{u}\end{equation}
satisfies $f(t) \sim \sqrt{2\pi}  u(0)t^{-1/2}$ as $t \to \infty$ with an explicit bound
\begin{equation}
\Big| f(t) - \sqrt{2\pi}  u(0) t^{-1/2}  \Big| \leq	C_k t^{-3/2} \| (1 + w^{2})^{-k} u'' \|_\infty,
\label{fu}\end{equation}
where $k$ is such that $\| (1 + w^{2})^{-k} u'' \|_\infty$ is finite, and $C_k$ is a constant depending only on $k$.

If $u(0) = 0$ then
$$
f(t) \sim \frac{\sqrt{2\pi}}{4} u''(0)t^{-3/2} ,
$$
with error bound
\begin{equation}
\Big| f(t) - \frac{\sqrt{2\pi}}{4} u''(0) t^{-3/2} \Big| \leq	C_k t^{-5/2} \| (1 + w^{2})^{-k} u'''' \|_\infty.
\label{fu''}\end{equation}
\end{lemma}

The proof is straightforward, and omitted.

\begin{proposition}\label{prop:heatkernellowerbounds}
Assume that $(X, g)$ is an asymptotically hyperbolic manifold with no eigenvalues and no resonance at the bottom of the spectrum. Then for $\lambda \in [0, \Lambda]$, $\Lambda > 0$ arbitrary,

(i) the resolvent kernel $R(-\imath \lambda)(z,z')$,  $z, z' \in X^\circ$, is
of the form $(\rho_L \rho_R)^{n/2 + \lambda} a(z,z'; \lambda)$, where $a$ is (or more precisely extends to be) in  $C^\infty \big( (X^2_0 \setminus \diag) \times [0, \Lambda] \big)$, and $a \geq \epsilon > 0$ uniformly down to $\lambda = 0$.

(ii) the derivative $\imath  (\partial_\lambda R)(0)(z,z')$ of $R(\lambda)$ at $\lambda= 0$ takes the form
$$
(\rho_L \rho_R)^{n/2} (-\log \rho_L - \log \rho_R) b(z,z'),$$ where $b \geq \epsilon > 0$ in a neighbourhood of the left and right faces $\FL$ and $\FR$.
\end{proposition}

\begin{proof}
The regularity statement that resolvent kernel lies in $(\rho_L \rho_R)^{n/2 + \lambda} C^\infty(X^2_0 \setminus \diag)$ follows from Mazzeo-Melrose's resolvent construction (Theorem \ref{thm : Mazzeo-Melrose resolvent}) and the assumptions of no eigenvalues or resonance at $\lambda = 0$. To show the strict positivity, we first do this in the case that $X$ is hyperbolic space $\mathbb{H}^{n+1}$. In that case, the \emph{weak} positivity, that is, $a \geq 0$, follows (ironically) from expressing the resolvent in terms of the heat kernel:
$$
R(-\imath \lambda) = \int_0^\infty e^{-t(\Delta_{X} - n^2/4)} e^{-t \lambda^2} \, dt \quad \lambda > 0.
$$
This shows that the Schwartz kernel of $R(-\imath  \lambda)$ is strictly positive on $(\mathbb{H}^{n+1})^2$ for positive $\lambda$, and hence, by continuity, nonnegative on $X^2_0$ for $\lambda \geq 0$. To show the strict positivity at the boundary of $X^2_0$, we note that the resolvent kernel is a function only of hyperbolic distance $r$, and in terms of $s = e^{-r}$, it satisfies the differential equation
$$
\bigg( -  \Big( s \frac{\partial}{\partial s} \Big)^2 + n \big( \frac{1 + s^2}{1-s^2} \big) s \frac{\partial}{\partial s} - n^2/4 \bigg) f(s) = 0,
$$
for $s \in (0,1)$ (where $s \to 0$ is the limit $r \to \infty$, corresponding to the left and right boundary hypersurfaces $\FL$ and $\FR$ on $X^2_0$). This is a regular singular ODE with a double indicial root $n/2$, i.e. as $s \to 0$ solutions have leading behaviour $s^{n/2}$ or $s^{n/2} \log s$. Since $R(-\imath \lambda)$ is regular at $\lambda = 0$, by Theorem~\ref{thm : Mazzeo-Melrose resolvent}, specifically  statement \eqref{MMregularity}, we must have $c s^{n/2}$ behaviour, that is, the logarithm cannot occur. (Note that $s = e^{-r}$ is comparable to $\rho_L \rho_R$ using \eqref{eqn : geodesic distance}.) Moreover, $c$ must be nonzero for the solution to be nontrivial. The nonnegativity already deduced shows that $c$ must be positive. Since $s = e^{-r}$ is comparable to $\rho_L \rho_R$ on the space $X^2_0$, implied by \eqref{eqn : geodesic distance},  $c > 0$ is equivalent to asserting that $a(z, z'; \lambda) > 0$ away from the diagonal. At the diagonal, $a(z,z', \lambda)$ tends to $+\infty$. Hence  by compactness, $a(0) \geq \epsilon$ globally on $X^2_0$. A similar argument applies to every $\lambda \geq 0$; the only difference is the indicial roots are $n/2 + \lambda$ and $n/2 - \lambda$, and \eqref{MMregularity} shows that we only get $s^{n/2 + \lambda}$ behaviour.

Next we consider the derivative of the resolvent at zero. Since the resolvent is, by assumption, holomorphic in a neighbourhood of $\lambda = 0$, we can differentiate the identity
$$
\Big( \Delta_{X} - \frac{n^2}{4} + \lambda^2 \Big) R(-\imath \lambda)(z,z') = \delta_{z'}(z)
$$
to obtain
$$
\Big( \Delta_{X} - \frac{n^2}{4} + \lambda^2 \Big) \big( -\imath  \frac{d}{d\lambda} R(\lambda) \big) \Big|_{\lambda = 0}(z,z') = 0.
$$
Next we claim that  $\imath  (\partial_\lambda R)(0)$ has a nonnegative kernel. This follows from relating the heat kernel to the resolvent. From \eqref{heatkernelfromres} we have a leading asymptotic for the heat kernel $H(t, z, z')$ as $t \to \infty$ for fixed $z,z'$:
$$
e^{-t\Delta_{\mathbb{H}^{n+1}}}(z,z) = e^{-n^2t/4} \frac{\imath }{2\pi} \int_{-\infty}^\infty e^{-t\lambda^2} R(\lambda)(z,z') \, 2\lambda d\lambda.
$$
According to Lemma~\ref{lem:fexp}, the leading order behaviour comes from the Taylor series of the resolvent at $\lambda = 0$, and the contribution of $R(0)$ vanishes due to oddness of the integral. So, following \eqref{fu''}, the leading order behaviour is
\begin{equation}
e^{-t\Delta_{X}}(z,z') \sim e^{-n^2t/4} t^{-3/2} \frac{\imath }{\sqrt{2\pi}} \big( \frac{d}{d\lambda} R \big) (0)(z,z').
\label{lob}\end{equation}
This leading asymptotic is necessarily nonnegative, so the nonnegativity of $i (\partial_\lambda R)(0)$ is established.

The maximum principle then implies that either the kernel of $\imath  (\partial_\lambda R)(0)$ is strictly positive, or identically zero. However, we have just seen that, in the special case of hyperbolic space, the kernel of $R(\lambda)$  is a function only of $s = e^{-r}$ and has the form  $g(\lambda,s) s^{n/2 + \lambda}$ for small $s$, where $g$ is smooth in $s$ and holomorphic in $\lambda$. Differentiating at $\lambda = 0$ we find that $\imath  (\partial_\lambda R)(0)$ has kernel of the form $g(0, 0) s^{n/2} \log s + O(s^{n/2})$, $s \to 0$, and $g(0,0)$ is the constant $c$ above which is strictly positive. It follows that the kernel of $\imath  (\partial_\lambda R)(0)$ is strictly positive in $({\mathbb{H}^{n+1}}^\circ)^2$, and indeed, bounded below by a positive multiple of $s^{n/2} \log s$ for $s  = e^{-r} \to 0$. Since $s$ is comparable to $\rho_L \rho_R$, this proves part (ii) of the proposition in the case of hyperbolic space.

We next prove the proposition for general asymptotically hyperbolic spaces $(X,g)$. We first note that the proof of strict positivity of the kernel of $R(-\imath  \lambda)$ in the interior works in general. Moreover, recall from Proposition~\ref{prop : normal = hyperbolic} and Theorem~\ref{thm : Mazzeo-Melrose resolvent} that the front face, $\FF$, of $X^2_0$ is fibred over the boundary $\partial X$, with fibres that have a natural hyperbolic structure, and in terms of this natural hyperbolic structure,  the kernel of $R(-\imath \lambda)$ agrees exactly with the hyperbolic resolvent, $(\Delta_{\mathbb{H}^{n+1}} - n^2/4 + \lambda^2)^{-1}$, there. In combination with the fact that the kernel of the resolvent $R(-\imath  \lambda)$, $\lambda \geq 0$, takes the form $(\rho_L \rho_R)^{n/2 + k}$ times a function that is continuous  away from the diagonal, and tends to $+\infty$ at the diagonal, this means that on $X^2_0$ the kernel is $(\rho_L \rho_R)^{n/2 + k}$ times a function $a$ that is positive in a neighbourhood of $\FF$. So it remains to show that $a$ is positive at $\FL$ and $\FR$ (outside a small neighbourhood of $\FF$).

We first show that it is positive in the interior of $\FR$. We note that ${x'}^{-n/2-\lambda}$ times the kernel of $R(-\imath \lambda)$ restricts to a smooth function $P(z, y')$ on $\FR$, which satisfies the equation $(\Delta_X - n^2/4 + \lambda^2) P( \cdot, y') = 0$ for each $y'$. Moreover, $P \geq 0$, and from the previous paragraph, it is strictly positive near $\FF$. By the maximum principle, $P$ is strictly positive on $\FR$, except possibly in the interior of $\FR \cap \FL$.

Near $\FR \cap \FL$, we use coordinates $(x, y, y')$. We want to show that for each $(y, y')$ with $y \neq y'$, $\tilde P(x, y, y')  := x^{-(n/2 + \lambda)} P(x, y, y')$ is bounded away from zero as $x \to 0$.

To do this, choose $(y_0, y'_0)$ with $y_0 \neq y_0'$, and we show that $\tilde P(-\imath \lambda)(0, y_0, y_0') \neq 0$. Let $u = u_{y_0'} = P(-\imath \lambda)(\cdot, y_0')$. Then $u$ is a `plane wave', i.e. satisfies $(\Delta_X - n^2/4 + \lambda^2)u = 0$. Moreover, since the kernel of the resolvent is $(\rho_L \rho_R)^{n/2 + \lambda}$ times a continuous function (away from the diagonal), according to Proposition~\ref{thm : Mazzeo-Melrose resolvent}, it follows that $u$ extends to the manifold $X_{y_0'} = [X; (0, y_0')]$ with the boundary point $(0, y_0')$ blown up,
to be $\rho_F^{-n/2 - \lambda} \rho_R^{n/2 + \lambda}$ times a continuous function. (Here we are writing $\rho_R$ and $\rho_F$ for the boundary defining functions of $X_{y_0'}$, with the latter defining the blowup face. We will write $\F$ and $\R$ for the corresponding boundary hypersurfaces.)
We need to show that, near $y= y_0$, we have $u = x^{n/2 + \lambda} g(x,y)$ with $g(0,y_0) > 0$.

Observe that we have, for any $\tilde \lambda > \lambda$,
$$
(\Delta_X - n^2/4 + \tilde\lambda^2) u = \epsilon u \geq 0, \quad \epsilon = \tilde \lambda^2 - \lambda^2 > 0.
$$
So, at least formally, we have
$$
u = \epsilon R(-\imath \tilde \lambda) u.
$$
However, it is not obvious that we can apply the resolvent $R(-\imath \tilde \lambda)$ to $u$; let us check this carefully.

\begin{lemma}\label{lem:epsilon} For each $z \in X^\circ$, we have
\begin{equation}
u(z) = \epsilon \int_{X^\circ} R(-\imath \tilde \lambda)(z,z') u(z') dg(z')
\label{uR}\end{equation}
where this integral converges uniformly for each $z \in X^\circ$.
\end{lemma}

We prove Lemma~\ref{lem:epsilon} at the end of this section. Accepting this lemma, then we can complete the proof of the positivity of $\tilde P(-\imath \lambda)$. We express
\begin{equation}
u = \epsilon R(-\imath  \tilde \lambda) u.
\label{uintermsofu}\end{equation}
Now we use the fact, established above, that for $z'$ in a compact set $K \subset X^\circ$ of the interior of $X$, we have
\begin{equation}
R(-\imath  \tilde \lambda)(x, y_0, z') = x^{n/2 + \tilde\lambda} \tilde R(x, y_0, z') \text{ where } R(0, y_0, z') > 0 \text{ for all } y_0 \in \partial X, z' \in K.
\label{tildeR}\end{equation}
This is just a restatement of the fact that the Poisson kernel $P(-\imath \tilde \lambda)(z,y')$ is strictly positive for $z \in K$, together with the symmetry of the resolvent kernel: $R(-\imath \tilde \lambda)(z,z') = R(-\imath \tilde \lambda)(z',z)$. Choose $K$ arbitrarily; by the strict positivity of $u$ in the interior of $X$, $u$ is bounded below by a positive constant on $K$. Then \eqref{uintermsofu} and \eqref{tildeR} combined, together with the nonnegativity of $u$ and the kernel of $R(-\imath \tilde \lambda)$, show that
\begin{equation}
u \geq c x^{n/2 + \tilde \lambda}, \quad x \to 0.
\label{uypos}\end{equation}
This is not quite what we want, since $\tilde \lambda > \lambda$. However, we claim that
\begin{equation}
u(x, y_0) = x^{n/2 + \lambda} g(y_0) + O(x^{n/2 + \lambda + 1}), \quad y_0 \neq y_0'.
\label{uclaim}\end{equation}
The proof of this claim is deferred to the end of this section.
Together with \eqref{uypos} this shows that, for $\epsilon$ sufficiently small, so that $\tilde \lambda < \lambda + 1$, we have   $g(y_0) > 0$, that is, $\tilde P(-\imath \mu)(0, y_0, y_0') > 0$, completing the proof of (i) for general $X$.

The proof of (ii) for general $X$ is now straightforward. We have shown that the kernel of $R(-\imath \lambda)$ takes the form
$(\rho_L \rho_R)^{n/2 + \lambda} a(\lambda)$, where $a$ is holomorphic in $\lambda$ as a smooth function on $X^2_0 \setminus \diag$; moreover $a$ is positive and bounded away from zero. So the kernel of $\imath  (\partial_\lambda R)(0)$ is of the form
$$
(\rho_L \rho_R)^{n/2} (-\log \rho_L - \log \rho_R) a(0) + O((\rho_L \rho_R)^{n/2} ),
$$
as either $\rho_L \to 0$ or $\rho_R \to 0$, that is, near $\FL \cup \FR$. This proves that in a neighbourhood of $\FL \cup \FR$, the kernel of $\imath  (\partial_\lambda R)(0)$ is bounded below by a positive multiple of $(\rho_L \rho_R)^{n/2} (-\log \rho_L - \log \rho_R)$. Restricted to $\FF$, the kernel is the same as for the case of hyperbolic space so we also have strict positivity near $\FF$. Finally, in the interior, nonnegativity follows from \eqref{lob} just as for hyperbolic space, and then the maximum principle shows that it must be strictly positive. A global positive lower bound follows from compactness of $X^2_0$.
\end{proof}

We also need the positivity in the limit $\lambda \to \infty$, that is, of the semiclassical resolvent at $h=0$. For this purpose, we write $\lambda = \sigma/h$, where $h \geq 0$ and $|\sigma| = 1$. Let us write $\tilde R_{od}(z, z', \sigma, h) = e^{\imath  \sigma r/h} R(z, z', \sigma/h)$ (here $e^{\imath  \sigma r/h} = e^{\imath \lambda r}$). We have already shown that the kernel of $\tilde R_{od}$ extends to a function on the semiclassical double space crossed with the half-circle $U_- := \{ \sigma \in \CC \mid |\sigma| = 1, \Im \sigma \leq 0 \}$, with regularity
$$
e^{\imath \sigma r/h} \rho_\mathcal{L}^{n/2} \rho_\mathcal{R}^{n/2} \rho_\mathcal{A}^{-n/2 +1} \rho_\mathcal{S}^{-n+1} \mathcal{A}^0\Big(X_0^2 \times_1 [0, 1)_h \times U_-\Big)
$$
in the region  $\{r/h  \geq 1 \}$ (that is, away from the diagonal).

\begin{proposition}\label{prop:reshighenergy} The kernel of $\tilde R_{od}(\cdot, \cdot, -\imath , h)$ near the boundary hypersurface $\mathcal{A}$, i.e. that boundary hypersurface at $h=0$ and away from the diagonal, is given by
$$
\tilde R_{od}(\cdot, \cdot, -\imath , h) = c h^{-n/2 + 1} g(r, \theta)^{-1/4} + O(h^{-n/2 + 2} e^{-nr/2})$$
in terms of normal polar coordinates $(r, \theta)$ on the left copy of $X$ (here $\theta$ is a variable in the sphere $S^n$) based at $z'$. Here $g$ is the absolute value of the determinant of the metric $g_{ij}$ written in the normal polar coordinates $(r, \theta)$, and
$c$ is a positive constant.

\end{proposition}

\begin{remark}\label{rem:g} On hyperbolic space, the function $g(r, \theta)$ is equal to $(\sinh r)^{2n}$. On an asymptotically hyperbolic Cartan-Hadamard manifold, the function $g(r, \theta)$ is uniformly comparable to $(\sinh r)^{2n}$ as shown in  \cite[Lemma 29]{Chen-Hassell2}.
\end{remark}

\begin{proof} This comes directly from the parametrix construction in \cite[Section 5, p.492]{Melrose-Sa Barreto-Vasy} together with Appendix~\ref{sec:resolvent from parametrix}.  If $E(z, z', \sigma, h)$ is their parametrix, then $h^{n/2 - 1}e^{r/h} E( \cdot, \cdot, -\imath , h)$ restricts to the semiclassical face $\mathcal{A}$, and satisfies the equation
$$
\partial_r (g(r, \theta)^{1/4} \cdot ) = 0,
$$
which has the unique solution $c g(r, \theta)^{-1/4}$. We have $c \geq 0$ due to nonnegativity of the resolvent kernel on the negative imaginary axis, and by matching with the leading behaviour we find that $c > 0$.

Then when we correct the parametrix to the true resolvent, the expansion at $\mathcal{A}$ is not affected (the correction term vanishes to all orders there), showing that $h^{n/2 - 1} \tilde R_{od}(z, z', -\imath , h)$ is also given by $c g(r, \theta)^{-1/4}$ at $\mathcal{A}$. This is uniformly comparable to $c e^{-nr/2}$ as noted in the remark above. The next term in the expansion is at order $h^{-n/2 + 2}$ and vanishes to order $n/2$ at the left and the right boundaries; since $(\rho_L \rho_R)^{n/2}$ is comparable to $e^{-nr/2}$, this proves the statement in the proposition.
\end{proof}

\begin{proof}[Proof of Lemma~\ref{lem:epsilon}] If we apply $\Delta_{z'} - n^2/4 + \tilde\lambda^2$ to the resolvent kernel $R(-\imath \lambda)(z,z')$, where the derivatives act in the $z'$ variable, we get the kernel of the identity operator, that is, the delta function along the diagonal. So we have, for fixed $z \in X^\circ$ and $\delta$ sufficiently small (so that the region $\{ x \geq \delta \}$ includes the point $z$),
\begin{equation}\begin{gathered}
u(z) = \int_{x' \geq \delta} \Big( (\Delta_{z'} - n^2/4 + \tilde\lambda^2)R(-\imath \tilde\lambda)(z,z') \Big) u(z') dg(z')
\\ = \lim_{\delta \to 0}  \int_{x' \geq \delta} \Big( (\Delta_{z'} - n^2/4 + \tilde\lambda^2)R(-\imath \tilde\lambda)(z,z') \Big) u(z') dg(z').
\end{gathered}\end{equation}
(Here we need to understand the kernel $R(-\imath \tilde \lambda)$ as a distribution near the singularity at $z' = z$. But away from this point, the kernel is smooth and we can interpret the derivatives in the classical sense.)

If we formally integrate by parts, then the Laplacian $\Delta_{z'}$ moves over to act on the function $u$. This clearly gives us the identity \eqref{uR} we seek, since $(\Delta_X - n^2/4 + \tilde\lambda^2) u = \epsilon u $. So we need to justify moving the derivatives over in the limit $\delta \to 0$; that is, we need to show that the boundary contribution, incurred in the integration-by-parts employed to shift derivatives from the resolvent $R(-\imath \tilde\lambda)$ to the function $u$, vanishes in the limit $\delta \to 0$.

The boundary terms take the form (where we now use coordinates $(x', y')$ in place of the right variable $z'$)
\begin{equation}\begin{gathered}
\delta \int_{x' = \delta} \bigg(  (x' \partial_{x'} R(-\imath \tilde \lambda)(z, x', y')) u(x', y') - R(-\imath \tilde \lambda)(z, x', y') (x' \partial_{x'} u(x', y'))  \\
+  n R(-\imath \tilde \lambda)(z,x',y') u(x',y') \bigg) \frac{ a dy'}{x'^{n+1}}.
\end{gathered}\end{equation}
The factor of $\delta$ outside the integral arises from the fact that, in the local coordinate expression for the Laplacian near $x' = 0$,  each $x'$ derivatives comes with a factor of $x'$ in front. This prefactor of $\delta$ is just such a factor of $x'$, evaluated at $x' = \delta$.

We now analyze this integral in the limit $\delta \to 0$. For $z$ fixed, using the second statement in Theorem~\ref{thm : Mazzeo-Melrose resolvent}, both the kernel $R(-\imath \tilde\lambda)(z, \cdot)$  as well as $x' \partial_{x'} R(-\imath \tilde\lambda)(z, \cdot)$ are $O({x'}^{n/2 + \tilde\lambda})$ as $x' \to 0$. On the other hand, $u$ is $O(\rho_F^{-n/2-\lambda} \rho_R^{n/2+\lambda})$ as $x' \to 0$. Notice the completely different behaviour (growth as opposed to decay) as $x' \to 0$ depending on whether $y' = y'_0$ or $y' \neq y'_0$.

So, for the variable $y'$ outside any neighbourhood of $y_0'$, things are straightforward: we can estimate the integral by
$$
\delta \int \delta^{n/2 + \tilde\lambda} \delta^{n/2 + \lambda} \frac{dy'}{\delta^{n+1}} = O(\delta^{\lambda + \tilde\lambda}) \to 0.
$$
The more delicate case is when $y'$ is close to $y'_0$, in which case the function $u$ is large there. Using boundary defining functions $\rho_F = \sqrt{{x'}^2 + (y' - y'_0)^2}$ for $\F$ and $x'/\rho_F$ for $\R$, we can estimate $u$ by
$C \rho_F^{-n/2 - \lambda} (x'/\rho_F)^{n/2 + \lambda} = C {x'}^{n/2 + \lambda} \rho_F^{-n-2\lambda}$ in this region. Then the boundary terms are estimated by
$$
C\delta \int \delta^{n/2 + \tilde\lambda} \frac{\delta^{n/2 + \lambda}}{(\delta^2 + |y' - y'_0|^2)^{n/2 + \lambda}} \frac{dy'}{\delta^{n+1}} = C\delta^{\tilde\lambda - \lambda} \int \Big( 1 + \frac{|y' - y'_0|^2}{\delta^2} \Big)^{-(n/2 + \lambda)} \frac{dy'}{\delta^{n}} = O(\delta^{\tilde\lambda - \lambda}) \to 0,
$$
where we use $\tilde\lambda > \lambda$ in the last step.
In each region, the boundary contribution vanishes. So we can take the limit $\delta \to 0$ and obtain \eqref{uR}.
\end{proof}

\begin{proof}[Proof of \eqref{uclaim}] This follows from a slightly more precise description of the resolvent kernel as a polyhomogeneous conormal function (away from the diagonal) on $X^2_0$, in the sense of \cite{Melrose-conormal}. In fact, the resolvent kernel $R(-i\lambda)$, outside any neighbourhood of the front face $\FF$, has the form $(xx')^{n/2 + \lambda} C^\infty(X^2)$. This is rather clear by construction for the \emph{parametrix} constructed by Mazzeo and Melrose; see \cite[Section 4]{Chen-Hassell1} for a proof that the correction term (the difference between the true resolvent and the parametrix) also has this form. (Although \cite{Chen-Hassell1} as stated only applies to the resolvent on the spectrum, this part of the argument applies equally to all values of the spectral parameter.)  It follows that $u = u_{y_0'}$ takes the form \eqref{uclaim} for any $y_0 \neq y_0'$.
\end{proof}


\section{Lower bound on the heat kernel}\label{sec:lower bound}
We now prove the Davies-Mandouvalos lower bound on the heat kernel under the same geometric assumptions.
\begin{proposition}\label{prop : lower}
The heat kernel obeys $$ e^{- t \Delta_X}(z, z')   \geq c t^{-(n + 1)/2} e^{- n^2t/4 - r^2/(4t) - nr/2}  (1 + r + t)^{n/2 - 1}(1 + r)$$
for some $c > 0$, where $r$ is the geodesic distance between $z$ and $z'$.
\end{proposition}

\begin{proof}
We consider different regions as in the proof of the upper bound. We remind the reader that these regions are
\begin{itemize}
\item[(i)] $t \geq C, r^2 \leq C$;
\item[(ii)] $t \geq C, \sqrt{C} \leq r \leq Ct$;
\item[(iii)] $t \geq C$, $r \geq Ct$;
\item[(iv)] $t \leq C, r^2 \geq C_1t$;
\item[(v)] $t \leq C_2, r^2 \leq C_3t$.
\end{itemize}

\subsection*{Region (i)} In this region, we need to prove an lower bound of the form $c e^{-n^2t/4} t^{-3/2}$. As in the proof of Proposition~\ref{prop:heatkernellowerbounds},  we get an expansion for the integral \eqref{w-integral} in powers of $t$, as $t \to \infty$, depending on the derivatives of the resolvent at $\lambda = 0$, with the leading term given by \eqref{lob}. The lower bound for sufficiently large $t$ then is a direct consequence of the positive lower bound on $i (\partial_\lambda R)(0)(z,z')$ for bounded distances from part (ii) of  Proposition~\ref{prop:heatkernellowerbounds} (bounded distance is equivalent to being outside some neighbourhood of $\FL$ and $\FR$ in $X^2_0$).

\subsection*{ Region (ii)} In this region, we need to show a lower bound of the form $e^{-n^2t/4} e^{-r^2/4t} t^{-3/2} r e^{-nr/2}$. Referring to the proof of the upper bound, we obtained an expression for $e^{n^2 t/4} e^{r^2/4t} e^{-t \Delta_X}$ as a sum of three terms in \eqref{w-integral-region ii}. The next line shows that of these terms, the first dominates, for large enough $r$. Therefore, it suffices to get a lower bound on this first term, namely,
$$
 \frac{r}{2t} \int_{-\delta/2}^{\delta/2} e^{-tw^2} R_{od}\big(w - (\frac{\imath r}{2t}) \big)    \, dw .
 $$
 Again using \eqref{lob}, the leading contribution to the integral
 $$
 \int_{-\delta/2}^{\delta/2} e^{-tw^2} R_{od}\big(w - (\frac{\imath r}{2t}) \big)    \, dw
 $$
 is $\sqrt{2\pi}  t^{-1/2} R_{od}(-\imath r/2t)(z, z')$, and by comparing \eqref{eqn : resolvent decomposition} and part (i) of  Proposition~\ref{prop:heatkernellowerbounds} (noting that $e^{-\lambda r}$ is comparable to $(\rho_L \rho_R)^{\lambda}$), $R_{od}(-\imath \lambda)(z, z')$ is bounded below by $c e^{-nr/2}$ for large $r$, uniformly for $\lambda \in [0, C]$, yielding the desired lower bound.

\subsection*{ Region (iii)}  In this region, we need to show that \eqref{w-integral} is bounded below by a multiple of
$$
r^{n/2} e^{-nr/2} t^{-(n+1)/2}.
$$
Notice that in the proof of the upper bound, we found an expression for $e^{n^2 t/4} e^{r^2/4t} e^{-t \Delta_X}$
of the form
$$
\frac{\imath }{2\pi} \int_{-\infty}^\infty e^{-tw^2} R_{od}(z, z', -\imath r/2t + w) (w - \frac{\imath r}{2t}) \, dw.
$$
Note that we can apply Lemma~\ref{lem:fexp}. Indeed, by Corollary~\ref{coro : resolvent construction}, $R_{od}(z, z', \lambda)$ is bounded polynomially in $|\lambda|$ and has fixed order growth as $r \to \infty$. It follows, using the analyticity of $R_{od}$ and the Cauchy integral formula to bound $\lambda$-derivatives of $R_{od}$ in terms of $R_{od}$ itself, that $\lambda$-derivatives of $R_{od}$ also have the same order growth in $|\lambda|$ and fixed order growth as $r \to \infty$ (that is, independent of the number of derivatives). It follows that, as $t \to \infty$, there is a leading asymptotic given by \eqref{fu}. We now compute this leading contribution.

In region (iii), $r/2t$ is large, so we are in the semiclassical regime. Fixing $r/t$, and letting
\begin{equation}
\sigma(w) = \frac{-\imath r/2t + w}{|-\imath r/2t + w|}, \quad h_0 = \frac{2t}{r}, \quad h(w) = \big| \frac{-\imath r}{2t} + w \big|^{-1},
\label{sigmah}\end{equation}
then this expression reads
\begin{equation}
\frac{\imath }{2\pi} \int_{-\infty}^\infty e^{-tw^2} \tilde R_{od}(z, z', \sigma(w), h(w)) (w - \frac{\imath r}{2t}) \, dw.
\label{region iii} \end{equation}
Note that $\sigma = -\imath  + O(h_0w)$, $h(w) = h_0 + O(h_0^2w^2)$ and both are smooth functions of $w$.
In Proposition~\ref{prop:reshighenergy}, we showed that there is an expansion
$$
R_{od}(z, z', \sigma, h) = h^{-n/2 + 1} |g(r, \theta)|^{-1/4} + O(h^{-n/2 + 2}).
$$
We put this in \eqref{region iii} and use the fact that, as $t \to \infty$, the main contribution is from $w=0$. We find that the heat kernel in this region satisfies
$$
e^{n^2 t/4} e^{r^2/4t} e^{-t \Delta_X} \sim \frac1{\sqrt{2\pi}} t^{-1/2} \bigg( \big( \frac{r}{2t} \big)^{n/2} c g(r, \theta)^{-1/4} (1 + O(t/r)) \bigg) + O \big(t^{-3/2}  \big( \frac{r}{2t} \big)^{n/2} e^{-nr/2} \big).
$$
For sufficiently large $t$ and sufficiently small $t/r$, i.e. for $C$ sufficiently large, half the first term serves as a lower bound. Using Remark~\ref{rem:g}, this gives us a lower bound comparable to $r^{n/2} t^{-(n+1)/2} e^{-nr/2}$, as required.

\subsection*{ Region (iv)}  In this region, given by $t \leq C, r^2 \geq C_1t$, we have $r \geq Ct$ for $C_1$ large enough compared to $C$ (in fact $C_1 \geq C^3$ will do). We will also need the fact that $t^3/r^2 = (t/r^2) t^2$ is small, again provided $C_1$ is sufficiently large relative to $C$. Recall that $r$ may be large or small in this region.

As in region (iii), we need to show that the integral \eqref{region iii}
is bounded by a multiple of
$$
(1+r)^{n/2} e^{-nr/2}t^{-(n+1)/2}.
$$

Here, $t$ is not large so we cannot use Lemma~\ref{lem:fexp} directly. Instead, we change variable to $\tilde w = wh_0$, where $h_0 = 2t/r$. We notice that $\sigma = \sigma(w)$ and $h = h(w)$ given by \eqref{sigmah}  are, in fact, both smooth functions of $w h_0$. This is clear by writing
\begin{equation}
\sigma(w) = \frac{-\imath  + wh_0}{|-\imath  + wh_0|}, \quad h(w) = h_0 \big| -\imath  + wh_0 \big|^{-1}.
\label{sigmah2}\end{equation}

By a slight abuse of notation we denote these (new) functions of $w h_0 = \tilde w$ by $\sigma(\tilde w)$ and $h(\tilde w)$. Then the expression \eqref{region iii} becomes
\begin{equation}
\frac{\imath }{2\pi} \int_{-\infty}^\infty e^{-t\tilde w^2/h_0^2} \tilde R_{od}(z, z', \sigma(\tilde w), h(\tilde w)) \frac{\tilde w}{h_0}  \, \frac{d \tilde w}{h_0} -  \frac{\imath }{2\pi} \int_{-\infty}^\infty e^{-t\tilde w^2/h_0^2} \tilde R_{od}(z, z', \sigma(\tilde w), h(\tilde w))   \frac{\imath r}{2t} \, \frac{d \tilde w}{h_0}.
\label{region iv} \end{equation}
Now we can apply Lemma~\ref{lem:fexp}, since $t/h_0^2 = r^2/4t \geq C_1/4$ is large in region (iv).
We apply \eqref{fu} to the first term; we find that this term is bounded by
$$
h_0^{-2} \big(\frac{t}{h_0^2} \big)^{-3/2} \sup_{\tilde w} \Big| (1 + \tilde w^2)^{-n/2 + 1} \partial_{\tilde w} \tilde R_{od}(z, z', \sigma(\tilde w), h(\tilde w)) \Big| .
$$
Notice that $\tilde w$ derivatives of $R_{od}$ are bounded by $h^{-n/2 + 1} e^{-nr/2} = h_0^{-n/2 + 1} (1 + \tilde w^2)^{n/2 - 1} e^{-nr/2}$, as is $R_{od}$ itself; this is a direct consequence of the conormality statement of \eqref{resolvent away from diag}. So the first term is bounded by a constant times
\begin{equation}
h_0^{-n/2+2} t^{-3/2} e^{-nr/2} = \big( \frac{t}{r^2} \big) \,  t^{-(n+1)/2} r^{n/2}  e^{-nr/2}.
\label{firstterm}\end{equation}
On the other hand, the second term in the integral  is given by
\begin{equation}\begin{gathered}
 \frac{r}{2t} \frac1{h_0} \Big(  \frac1{\sqrt{2\pi}}  \big( \frac{t}{h_0^2} \big)^{-1/2} \Big( h_0^{n/2-1} c g(r, \theta)^{-1/4} (1 + O(h_0)) \Big)
+ O\Big( \big( \frac{t}{h_0^2} \big)^{-3/2} e^{-nr/2} \Big) \\
=  t^{-(n+1)/2} r^{n/2}   \frac1{\sqrt{2\pi}}  c g(r, \theta)^{-1/4} (1 + O(h_0))
+ O\Big( \big( \frac{t}{r^2} \big) t^{-(n+1)/2} r^{n/2}   e^{-nr/2} \Big)
\end{gathered}\label{secondterm}\end{equation}
Since $t/r^2 \leq C_1^{-1}$ is small in this region, \eqref{firstterm} is of the same size as the error term in \eqref{secondterm}. So the leading contribution is given by the first term in \eqref{secondterm}. As  we have already noted, in Remark~\ref{rem:g}, $|g(r, \theta)|^{-1/4}$ is bounded below by a constant times $e^{-nr/2}$. So for sufficiently large $C$ and $C_1$, a lower bound is given by half the first term, and this establishes the desired lower bound in this region.

\subsection*{ Region (v)}
In this region, the Gaussian decay term is bounded away from zero, so we only need to prove a lower bound of $t^{-(n+1)/2}$. This follows from Cheeger-Yau's lower bound, namely Theorem \ref{thm : cheeger-yau}, by comparing with the heat kernel on a space with constant curvature $-K$, where this is a lower bound for the sectional curvature on $M$.

\end{proof}


\section{Gradient estimates}\label{sec:gradient}
In this section, we prove estimates on the time and space gradient of the heat kernel. We start with time derivative estimates, which are relatively straightforward as the time
derivative of the heat kernel remains in the functional calculus of the Laplace operator, and therefore can be obtained exactly as in Section~\ref{sec:upper bound}. We find that, under the same assumptions as in Theorem~\ref{thm : heat kernel bounds}, an upper bound for the heat kernel is furnished by the Davies-Mandouvalos quantity multiplied by $C(1 + t^{-1} + r^2 t^{-2})$.

\begin{proposition}\label{prop : grad}
Let $X$ be an $n+1$-dimensional asymptotically hyperbolic Cartan-Hadamard manifold with  no eigenvalues and no resonance at the bottom of the spectrum.
Then the heat kernel obeys 
$$ 
\Big| \frac{\partial}{\partial t} e^{- t \Delta_X}(z, z') \Big|  \leq C  t^{-(n + 1)/2} e^{- n^2t/4 - r^2/(4t) - nr/2}  (1 + r + t)^{n/2 - 1}(1 + r) \big(1+  \frac1{t} + \frac{r^2}{t^2} \big), $$ where $r$ is the geodesic distance between $z$ and $z'$.
\end{proposition}

\begin{proof}
We break up the proof into the same five regions as in the proof of Proposition~\ref{prop : upper}. 
The estimates are straightforward adaptations of those in the proof of the upper bound, due to the following simple observation: the time derivative of the heat kernel can be expressed in terms of the spectral measure according to 
\begin{equation}
\frac{\partial}{\partial t} e^{-t(\Delta_X)}  = -\frac{n^2}{4} e^{-t(\Delta_X)} + \frac{\imath}{2\pi} e^{-tn^2/4} \, \int_{0}^\infty e^{-t\lambda^2}dE_{\sqrt{\Delta_X- n^2/4}}(\lambda) \,   2\lambda^3  \, d\lambda   
\label{eqn : heat kernel expression sm gradient}\end{equation}
or the resolvent:
\begin{equation}\begin{gathered}
\frac{\partial}{\partial t} e^{-t(\Delta_X - n^2/4)} = -\frac{n^2}{4} e^{-t(\Delta_X)} +  \frac{\imath }{2\pi} \int_{-\infty}^\infty e^{-t\lambda^2} R(\lambda - \imath 0) \, 2\lambda^3 \, d\lambda.
\end{gathered}\label{heatkernelfromres}\end{equation}
Compared to Section~\ref{sec:upper bound} we have an additional factor of $\lambda^2$ in the integrals. We now work our way through the estimates to see that the effect of this extra factor is, at worst, an additional factor $t^{-1} + r^2 t^{-2}$ in the upper bound.

\subsection*{Regions (i) and (v)} In region (i), it suffices to prove an upper bound of the form $e^{-n^2t/4} t^{-5/2}$. Applying \eqref{eqn : upper bound spectral measure},  we obtain 
\begin{equation}
\int_0^\infty e^{-t\lambda^2} \lambda^2 dE_{\sqrt{\Delta_X- n^2/4}}(\lambda)   \leq C \int_0^1 e^{-t\lambda^2} \lambda^4 \, d\lambda
+ C \int_1^\infty e^{-t\lambda^2} \lambda^{n+2}  \, d\lambda\leq C t^{-5/2},
\label{spec-meas-grad-est}\end{equation}
for $t \geq C > 0$, as required. The estimate in region (v) works in exactly the same way, except that the main contribution comes from the second integral, where $\lambda \geq 1$.

\subsection*{ Regions (ii), (iii), (iv)} In these regions, we shift the contour as before, obtaining
\begin{equation}\label{w-integral-gradient}
\frac{\imath }{2\pi} \int_{-\infty}^\infty e^{-tw^2} R_{od}\big(w - (\frac{\imath r}{2t})\big)(z,z')  \big(w - (\frac{ \imath r}{2t}) \big)^3 \, dw. 
\end{equation}
We need to show that this is bounded by $t^{-1} + r^2 t^{-2}$ times the upper bound of Section~\ref{sec:upper bound}. All these estimates work in a similar way. We have an extra factor of $(w - \imath r/(2t))^2$ in the integrand. We expand this to $w^2 - \imath wr/(2t) - r^2/(4t^2)$. Due to the scaling in the integrand, each additional factor of $w$ in the integrand yields an extra $t^{-1/2}$. So we gain an additional factor of $t^{-1} + r t^{-1/2} + r^2 t^{-2}$ (up to constants). Using the elementary inequality $2ab \leq a^2 + b^2$ we can eliminate the middle term and we get an upper bound of $t^{-1} + r^2 t^{-2}$ times the upper bound from Section~\ref{sec:upper bound}, as required. 
\end{proof}

We next establish spatial gradient estimates for the heat kernel. These follow directly from the time derivative estimate and the  Li-Yau gradient inequality: 

\begin{theorem}[Li-Yau]Let M be a $k$-dimensional complete boundaryless manifold with the Ricci curvature bounded from below by $-K \leq 0$. Suppose $u(z, t)$ is
a positive solution on $M \times (0, T]$ of the homogeneous heat equation $$(\Delta_M - \frac{\partial}{\partial t}) u = 0.$$ Then we have following inequality\begin{equation}   \label{eqn:gradient estimates} \frac{|\nabla_z u|^2}{u^2} - \alpha \frac{\partial_t u}{u} \leq \frac{k\alpha^2}{2t}  + C \frac{K}{\alpha - 1} ,\end{equation} for any $\alpha \in (1, 2)$.\end{theorem}

Choosing $\alpha = 3/2$ arbitrarily, applying this to the heat kernel and rearranging gives 
\begin{equation}
\Big| \nabla_z e^{-t\Delta_X}(z,z') \Big| \leq C \bigg(  \sqrt{ \Big|e^{-t\Delta_X}(z,z') \Big| \, \Big|\partial_t  e^{-t\Delta_X}(z,z') \Big|} + \Big| e^{-t\Delta_X}(z,z') \Big|\sqrt{1 + \frac1{t}} \bigg)  .
\end{equation} 
Substituting in our estimates from Proposition~\ref{prop : upper} and Proposition~\ref{prop : grad}, we obtain 

\begin{proposition}\label{prop : spatgrad}
Let $X$ be an $n+1$-dimensional asymptotically hyperbolic Cartan-Hadamard manifold with  no eigenvalues and no resonance at the bottom of the spectrum.
Then the spatial gradient of the heat kernel obeys 
$$ 
\Big| \nabla_z e^{- t \Delta_X}(z, z') \Big|  \leq C  t^{-(n + 2)/2} e^{- n^2t/4 - r^2/(4t) - nr/2}  (1 + r + t)^{n/2 - 1}(1 + r) \big( 1 + \frac{r}{t^{1/2}} + t^{1/2} \big), $$ where $r$ is the geodesic distance between $z$ and $z'$.
\end{proposition}


\section{Riesz transform}\label{sec:Riesz}
Auscher, Coulhon, Duong and Hofmann have proved the following implication for the Riesz transform on manifolds with exponential volume growth. 

\begin{theorem}[\cite{Coulhon-Duong, Auscher-Coulhon-Duong-Hofmann}]\label{thm : ACDH}Suppose $M$ is a complete Riemannian manifold being of doubling growth for small balls and exponential growth for large balls, that is \begin{eqnarray*}|B_{2r}(z)| \leq C |B_r(z)| && \mbox{for any $z \in M$ and $0 < r < 1$};\\|B_{\theta r}(z)| \leq C e^{c\theta} |B_r(z)| && \mbox{for any $z \in M$ and $1 \leq r < \infty$}.\end{eqnarray*}If the Laplacian has $\Delta_M$ a spectral gap $\lambda > 0$, i.e. $$\lambda\|f\|_{L^2} \leq \|\Delta_M f\|_{L^2} \quad \mbox{for any $f\in C_c^\infty(M)$}$$ and the heat kernel $H_M$ obeys the short time diagonal estimates and the short time gradient estimates, i.e. 
\begin{eqnarray}\label{eqn : diagonal estimates}H_M(t, z, z) \leq \frac{C}{|B_{\sqrt{t}}(z)|} \quad \mbox{for any $z\in M$ and $0 < t < 1$},\\ \label{eqn : short time gradient estimates} |\nabla_z H(t, z, z')| \leq C \frac{1}{\sqrt{t}|B_{\sqrt{t}}(z')|} \quad \mbox{for any $z, z'\in M$ and $0 < t < 1$},
\end{eqnarray} then the Riesz transform $\nabla \Delta_M^{-1/2}$ is $L^p$-bounded for $1 < p < \infty$.\end{theorem}

The short-time estimates \eqref{eqn : diagonal estimates}, \eqref{eqn : short time gradient estimates} follow easily from the upper bounds proved above. In fact, we can obtain these estimates without making any spectral assumptions on the manifold $X$. Therefore, our estimates imply boundedness of the Riesz transform $\nabla \Delta_X^{-1/2}$ on all $L^p$ spaces, $1 < p < \infty$.  However, this result is covered by a 1985 theorem of Lohou\'e in the more general context of Cartan-Hadamard manifolds with bounds on derivatives of order up to 2 of the curvature \cite{Lohoue}, so we omit the details. 

One could also remove the spectral gap, and ask whether $\nabla (\Delta_X - n^2/4)^{-1/2}$ is bounded on $L^p$. However, a moment's thought shows that, even on $L^2$, this is not bounded, as it is equivalent to the boundedness of $\Delta_X (\Delta_X - n^2/4)^{-1} = \Id + n^2/4 (\Delta_X - n^2/4)^{-1}$, which is clearly unbounded. On the other hand, the same reasoning shows that $\nabla (\Delta_X - n^2/4 + \lambda^2)^{-1/2}$ is bounded at least on $L^2$, for all $\lambda > 0$.
\footnote{This is related to work of Clerc-Stein \cite{Clerc-Stein}, who showed that a necessary condition for $L^p$ boundedness of functions $F(\Delta_{\HH^{n+1}} - n^2/4)$  is that $F$ extends to a holomorphic function in a strip; thus $(\Delta_{\HH^{n+1}} - n^2/4)^{-1/2}$ cannot act boundedly on $L^p$, but $(\Delta_{\HH^{n+1}} - n^2/4 + \lambda^2)^{-1/2}$  will for some range of $p$. See also Taylor \cite{Taylor}}
This motivates the question of finding the range of $p$ for which $\nabla (\Delta_X - n^2/4 + \lambda^2)^{-1/2}$ is bounded on $L^p$, for $\lambda \in (0, n/2)$. 

To accomplish this, we use another result which is obtained from \cite{Auscher-Coulhon-Duong-Hofmann}:

\begin{theorem}[\cite{Auscher-Coulhon-Duong-Hofmann}]\label{thm:ACDH2} Let $M$ be a complete noncompact manifolds with at most exponential volume growth of balls, and satisfying the local Poincar\'e property, that is, for all $r_0 > 0$ there is a constant $C_{r_0}$ such that, for every ball $B$ of radius $\leq r_0$, and every $f \in H^1(B)$, we have
$$
\int |f - f_B|^2 \, dg \leq C_{c_0} r^2 \int_B |\nabla f|^2 \, dg, \quad f_B = \int_B f \, dg. 
$$
If for some $p_0 > 2$, some $\alpha \in \RR$ and some $C$ independent of $t > 0$ we have
\begin{equation}\label{eq:gradestp0}
\| \nabla e^{-t\Delta_M} \|_{L^{p_0} \to L^{p_0}} \leq \frac{C}{\sqrt{t}} e^{-\alpha t}, \quad \text{for all } t > 0, 
\end{equation}
and the bottom of the spectrum of $\Delta_M$ is $b > 0$, then for all $a < \min(\alpha, b)$ the Riesz transform $\nabla (\Delta - a)^{-1/2}$ is bounded on $L^p$ for $p \in [2, p_0)$. 
\end{theorem}
This is essentially their Theorem 1.6, but it is not quite as stated in \cite{Auscher-Coulhon-Duong-Hofmann}; we have modified the statement to allow for negative values of $\alpha$. Notice that the local Poincar\'e property is satisfied by asymptotically hyperbolic manifolds. 

This motivates studying estimates of the form \eqref{eq:gradestp0}. To help us do this we note the following consequence of the `Kunze-Stein phenomenon' for $L^p$ estimates on hyperbolic space:

\begin{proposition}\label{prop:KS}
Let $T(z,z')$ be an integral kernel on hyperbolic space $\HH^{n+1}$ depending only on hyperbolic distance: $T(z,z') = f(d(z,z'))$. Let $p_0 > 2$ and let $z_0 \in \HH^{n+1}$. Then $T$ maps from $L^{p_0}(\HH^{n+1})$ to $L^{p_0}(\HH^{n+1})$ provided that the function $m(z) := f(d(z,z_0))$ is in $L^q(\HH^{n+1})$, for some $q < p_0'$, with an operator norm bound
$C(n,p_0, q) \| m \|_{L^q(\HH^{n+1})}$.
\end{proposition}

\begin{proof} This result is inspired by similar results in \cite{Anker-Pierfelice}, used to prove Strichartz estimates on hyperbolic space. 
We recall that $\HH^{n+1} =  G/K$ where $G = SO(n,1)$, $K = SO(n)$, and since $G$ here is semisimple, we have the Kunze-Stein phenomenon, that is, the convolution estimate \cite{Lipsman, Cowling, Herz, Cowling97}
\begin{equation}
L^2(G) * L^p(G) \subset L^2(G), \quad \| f*g \|_{L^2(G)} \leq C(n,q) \| f \|_{L^2(G)} \| g \|_{L^p(G)}.
\label{KS1}\end{equation}
We also have the trivial $L^1$ convolution inequality:
\begin{equation}
L^\infty(G) * L^1(G) \subset L^\infty(G), \quad \| f*g \|_{L^\infty(G)} \leq  \| f \|_{L^\infty(G)} \| g \|_{L^1(G)}.
\label{KS2}\end{equation}
Now we fix $g \in L^1(G) \cap L^2(G)$ and consider the operator of right convolution with $g$. Using \eqref{KS1}, \eqref{KS2} and Riesz-Thorin (or more precisely complex interpolation), we see that right convolution with $g$ maps  $L^{p_0}(G)$ to $L^{p_0}(G)$, with a bound of the form $C(n,q) \| g\|_{L^q(G)}$, provided  
$$
\frac1{q} = \frac{\alpha}{p} + \frac{1-\alpha}{1} \text{ and } \frac1{p_0} = \frac{\alpha}{2}, \quad \alpha \in [0,1], 
$$
that is, for
$$
\frac1{q} = \frac{\alpha}{2} + \frac{1-\alpha}{1} + \alpha \big( \frac1{p} - \frac1{2} \big) = 1 - \frac{\alpha}{2} + \alpha \big( \frac1{p} - \frac1{2} \big) =  \frac1{p_0'} + \alpha \big( \frac1{p} - \frac1{2} \big).
$$
This covers all $q$ in the range $[1, p_0')$ since $p$ can be taken arbitrarily close to $2$. 

 If $g$ is invariant under right multiplication by $K$, i.e. $g$ is really a function on hyperbolic space, then the same is true of the convolution. 
Finally, an integral operator on $\HH^{n+1} = G/K$ that depends only on distance may be viewed as a convolution by a function lifted from $G/K$ (in fact, a bi-invariant function). To see this, we lift the integral kernel $T$  to $G \times G$ and 
represent the action on a function $u$ on hyperbolic space as 
$$
\int \tilde T(g, g') \tilde u(g') \, dg'
$$
(up to a normalization factor) where $dg'$ is Haar measure on $G$, $\tilde T$ is the lift of $T$, and $\tilde u$ is the lift of the function $u$ to $G$. Since elements of $g$ act by left multiplication on $\HH^{n+1} = G/K$ as hyperbolic isometries, and $\tilde T$ depends only on the hyperbolic distance between $g$ and $g'$ (or more precisely their images in $\HH^{n+1} = G/K$), this is equal to 
$$
\int \tilde T({g'}^{-1} g, e) \tilde u(g') \, dg' = \tilde u * h, \quad h(g) = \tilde T(g, e).
$$
Then $h$ is just the lift of the function $m$ in the proposition. This completes the proof. 
\end{proof}

We use Proposition~\ref{prop:KS} to get operator norm estimates on the gradient of the heat kernel. 
We choose a global diffeomorphism mapping $X$ to $\BB^{n+1}$. For example, we can choose points in $X^\circ$ and in $\HH^{n+1}$, choose global normal coordinates based at each point, and then choose the map $\phi : X^\circ \to \HH^{n+1}$ that is the identity in these coordinates. This extends to a map between the compactifications, $X$ and $\BB^{n+1}$, respectively, and therefore induces a global diffeomorphism $\Phi$ from $X^2_0$ to $(\BB^{n+1})^2_0$ fixing the diagonal. Let $r$ be the geodesic distance on $X^2_0$ and $\tilde r$ on $(\BB^{n+1})^2_0$. Then $r^2$ is a quadratic defining function for the 0-diagonal in $X^2_0$ and, near the boundary hypersurfaces $\FL$ and $\FR$, we have $r = - \log(\rho_L \rho_R) + b$, where $b$ is bounded. Similar statements are true on 
$(\BB^{n+1})^2_0$, both for $\hat r$, the pullback of $r$ under $\Phi^{-1}$ to $(\BB^{n+1})^2_0$, and for $\tilde r$. As both $\hat r^2$ and $\tilde r^2$ are quadratic boundary defining functions for the 0-diagonal, they are comparable for small values. For large values, if we let $\rho_L, \rho_R$ be boundary defining functions for the left and right boundary hypersurfaces in $X^2_0$, and $\tilde \rho_L, \tilde \rho_R$ boundary defining functions for the left and right boundary hypersurfaces in $(\BB^{n+1})^2_0$, then under the diffeomorphism $\Phi^{-1}$,  $\rho_L \rho_R$ on $X$ pulls back to $a \tilde \rho_L \tilde \rho_R$ where $a$ is smooth and positive, so $a, a^{-1}$ are both bounded. That is, $-\log(\rho_L \rho_R)$ on $X$ pulls back to $-\log(\tilde \rho_L \tilde \rho_R) + \log a$, where $|\log a|$ is bounded. It  follows from this and Proposition~\ref{prop:dist} that there  are constants $C_1, C_2$ such that $\hat r$ satisfies 
\begin{equation}
C_1^{-1} \leq \frac{\hat r}{\tilde r} \leq C_1, \quad \tilde r \leq C_2; \quad \big| \hat r - \tilde r \big| \leq C_2/2, \quad \tilde r \geq C_2.
\label{rcomparison}\end{equation}
Now we estimate the operator norm of the gradient of the heat kernel. 

\begin{proposition}\label{prop: gradestLp} Let $X$ be an asymptotically hyperbolic Cartan-Hadamard manifold with no eigenvalues and no resonance at the bottom of the spectrum.  The gradient of the heat kernel on $X$ satisfies the following estimate on $L^p(X)$, for $p \in (2, \infty)$:
\begin{equation}
\| \nabla e^{-t\Delta_X} \|_{L^{p} \to L^{p}} \leq \frac{C}{\sqrt{t}} e^{-\alpha t}, \quad \text{for all } t > 0, \ \alpha < \frac{n^2(p'- 1)}{{p'}^2}, 
\end{equation}
where $C$ depends on $\alpha$ but not $t$. 
\end{proposition}

\begin{proof}
By Proposition~\ref{prop : spatgrad}, it is enough to estimate the operator norm of the integral operator $K$ on $X^\circ$ with kernel
$$
f(r, t)  = t^{-(n + 2)/2} e^{- n^2t/4 - r^2/(4t) - nr/2}  (1 + r + t)^{n/2 - 1}(1 + r) \big( 1 + \frac{r}{t^{1/2}} + t^{1/2} \big).
$$
Let 
$$
g(\tilde r,t) = \begin{cases} \max_{C_1^{-1} \leq \hat r/\tilde r \leq C_1} f(\hat r, t), \quad \tilde r \leq C_2, \\
\max_{| \hat r - \tilde r | \leq C_2/2} f(\hat r, t), \quad \tilde r \geq C_2.
\end{cases}
$$
It is not hard to check that, for some constant $C$, 
\begin{equation}
g(\tilde r, t) \leq C f(s(\tilde r), t), \quad s(\tilde r) = \begin{cases} C_1^{-1} \tilde r, \quad \tilde r \leq C_2, \\ \max(C_1^{-1} \tilde r, \tilde r - C_2/2), \quad \tilde r \geq C_2. \end{cases}
\label{rs}\end{equation}
This follows from the fact that the factor $e^{-r^2/4t - nr/2}$ is decreasing in $r$, and all other factors change at most by a constant if $r$ is scaled by a constant. 

To estimate the operator norm of $K$ on $L^{p_0}(X)$, we transfer the kernel to hyperbolic space. That is, we consider the operator $\phi \circ K \circ \phi^{-1}$ on hyperbolic space. As noted in \cite[Lemma 29]{Chen-Hassell2}, the distortion of the function $\phi$, that is, the ratio between $\phi_* dg$, the pushforward of the Riemannian measure on $X^\circ$, and the Riemannian measure on hyperbolic space, is uniformly bounded. Therefore, it acts as an isometry between the $L^{p_0}$ spaces of $X^\circ$ and $\HH^{n+1}$. So, in order to bound the operator norm of $K$ up to a uniform constant,  it is enough to bound the integral operator $\phi \circ K \circ \phi^{-1}$. Up to the distortion factors, which we ignore, this has kernel given by $K(z,z') \circ \Phi^{-1}$. Recall that $\hat r$ is the pullback of $r$ to $(\BB^{n+1})^2_0$. Given \eqref{rcomparison} and the definition of $g(\tilde r,t)$, we see that the pullback kernel $K(z,z') \circ \Phi^{-1}$ has kernel bounded by $g(\tilde r,t)$. Therefore, by Proposition~\ref{prop:KS}, the operator norm of the gradient of the heat kernel on $L^{p_0}(X)$ is bounded by a constant, independent of $t$, times 
$$ 
\| g(\cdot, t) \|_{L^q(\HH^{n+1})} = \bigg( \int_0^\infty |g(\tilde r,t)|^q (\sinh \tilde r)^n d\tilde r \bigg)^{1/q} \leq C  \bigg( \int_0^\infty |f(s,t)|^q (\sinh s)^n ds \bigg)^{1/q}.
$$
The last equality follows because, under  the change of variable from $r$ to $s$ in \eqref{rs}, the measure $(\sinh \tilde r)^n d\tilde r$ changes to $(\sinh s)^n ds$ up to a uniformly bounded factor. (Notice that it is crucial here that $s - r = O(1)$ when $s$ is large.)

The upshot is that, up to a constant independent of $t$, the operator norm of the gradient of the heat kernel on $L^{p_0}(X)$ is bounded by 
\begin{equation}
t^{-(n + 2)/2} e^{- n^2t/4} \bigg( \int_0^\infty e^{- qr^2/(4t) - qnr/2}  (1 + r + t)^{q(n/2 - 1)}(1 + r)^q \big( 1 + \frac{r}{t^{1/2}} + t^{1/2} \big)^q (\sinh r)^n \, dr \bigg)^{1/q}.
\label{q-expr}\end{equation}

We argue separately for $t \leq 1$ and $t \geq 1$. For $t \leq 1$, we set $q = 1$. Then the expression above is the formula for the hyperbolic heat kernel, multiplied by $(t^{-1/2} + r/t + t^{1/2})$. After bounding $(\sinh r)^n$ by $e^{nr}$ and collecting the exponential terms together as $e^{-1/4t(r - nt)^2}$, it is easy to check that \eqref{q-expr} is bounded by a multiple of $t^{-1/2}$.

For $t \geq 1$, we choose any $q$ less than $p_0'$, with the optimal exponential decay in time arising as $q \to p_0'$. To estimate the integral in \eqref{q-expr}, we replace $(1 + r + t)^{q(n/2 - 1)}$ by $C_q(1 + r)^{q(n/2 - 1)}t^{q(n/2 - 1)}$, $(\sinh r)^n$ by $e^{nr}$ and $(1 + \frac{r}{t^{1/2}} + t^{1/2})^q$ by $C_q (1+r)^q t^{q/2}$. Thus, we would like to estimate 
\begin{equation}
t^{-1/2} e^{- n^2t/4} \bigg( \int_0^\infty e^{- qr^2/(4t) + n(1-q/2)r}  (1 + r)^{q(n/2+1)} \, dr \bigg)^{1/q}.
\label{q-expr-2}\end{equation}
Completing the square inside the exponential and estimating, we find that 
$$
\int_0^\infty e^{- qr^2/(4t) + n(1-q/2)r}  (1 + r)^{q(n/2+1)} \, dr \leq C_\epsilon \exp\big( \frac{n^2 t (2-q)^2}{4q} + \epsilon \big)
$$
for any $\epsilon > 0$. Substituting into \eqref{q-expr-2} we obtain an operator norm estimate of the form 
$$
C_\epsilon \exp\big( -\frac{n^2t}{4} + \frac{n^2 t (2-q)^2}{4q^2} + \epsilon \big) = C_\epsilon \exp\big( -\frac{n^2(q-1)t}{q^2}  + \epsilon \big).
$$
  This proves Proposition~\ref{prop: gradestLp} as we can take $q$ arbitrarily close to $p'$. 
\end{proof}

The combination of Theorem~\ref{thm:ACDH2} and Proposition~\ref{prop: gradestLp} immediately implies Theorem~\ref{thm:Riesz}, which we restate as a corollary: 

\begin{corollary}\label{cor: Riesz}
The Riesz transform  $\nabla (\Delta_X - n^2(q-1)/q^2)^{-1/2}$ is bounded on $L^p(X)$ for all $p$ in the range $(q, q')$. Equivalently, for $\lambda \in (0, n/2)$, the Riesz transform  $\nabla (\Delta_X - n^2/4 + \lambda^2)^{-1/2}$
is bounded on $L^p$ for $p$ in the range $(2n/(n+2\lambda), 2n/(n-2\lambda))$, that is, for $p$ in the range \eqref{p-lambda-cond}. 
\end{corollary}

\begin{remark} Notice that the above range of exponents $p$ shrinks to $\{ 2 \}$ as $\lambda \to n/2$, and increases to the full range $(1, \infty)$ as $\lambda \to 0$. We expect  that this range is optimal, aside from possibly the endpoints. This is because the operator $(\Delta_X - n^2/4 + \lambda^2)^{-1/2}$ can be expressed as an integral over the resolvent, as in \cite{Guillarmou-Hassell}:
$$
(\Delta_X - n^2/4 + \lambda^2)^{-1/2} = \frac{2}{\pi} \int_0^\infty (\Delta_X - n^2/4 + \lambda^2 + \mu^2)^{-1} \, d\mu.
$$
We know that the kernel of the resolvent $(\Delta_X - n^2/4 + \lambda^2 + \mu^2)^{-1}$ decays at infinity (up to polynomial factors in $r$) as $e^{-(n/2 + \sqrt{\lambda^2 + \mu^2})r}$; hence the slowest decay is contributed at $\mu = 0$, where the decay is $e^{-(n/2 + \lambda)r}$. Composing with the gradient operator $\nabla$ cannot improve the rate of decay since it is exponential in $r$, and the radial component of the gradient is just $\partial_r$. 
This decay fails to be in $L^{p'}(X) \cap L^p(X)$ for $p$ outside the given range, hence we cannot expect this kernel to act on $L^p(X)$ for $p$ outside the given range. (Conversely, it follows from Taylor's work \cite{Taylor} that the operator $(\Delta_X - n^2/4 + \lambda^2)^{-1/2}$ is bounded on $L^p(X)$ for $p$ in the range \eqref{p-lambda-cond}, since the function $(z^2 + \lambda^2)^{-1/2}$ is holomorphic in an open strip of width $\lambda$.) 
\end{remark}

\begin{proof} The statement is trivial for $p=2$. For $p >2$, this follows directly from Theorem~\ref{thm:ACDH2}  and Proposition~\ref{prop: gradestLp}. 
The heat kernel gradient pointwise estimate in Proposition~\ref{prop : spatgrad} is symmetric, hence it applies equally well to the adjoint. We can repeat the argument above applied to the formal adjoint of the Riesz transform, showing that it, too,  is bounded on $L^p$ for $p > 2$ in the range above. Taking adjoints we obtain boundedness of the Riesz transform on $L^p$ for the dual exponents. 
\end{proof}

\begin{remark} One can adapt the argument from the original paper of Coulhon and Duong \cite[Proof of Theorem 1.3]{Coulhon-Duong} to show the boundedness of the Riesz transform for $\Delta_X - n^2/4 + \lambda^2$, $\lambda \in (0, n/2)$, on $L^p$ for  a range of $p$ below $2$. However, this argument gives a strictly smaller range than in Corollary~\ref{cor: Riesz} when $p < 2$, perhaps because one does not have a Kunze-Stein phenomenon to exploit in the very general context of \cite[Theorem 1.3]{Coulhon-Duong}. 
\end{remark}



 \section{Open problems}
 We return to the question asked after the statement of the main theorem, Theorem~\ref{thm : heat kernel bounds}, in the introduction. 
 
 \begin{question} Are there any asymptotically Euclidean metrics on $\RR^d$, not diffeomorphic to the standard flat metric, for which the heat kernel is bounded above and below by multiples of the Euclidean heat kernel \eqref{eqn:heat kernel euclidean}, uniformly for all distances and all times? 
\end{question}

We do not know the answer to this question. However,   in dimension $d=2$, assuming that the curvature is an integrable function on $\RR^2$, the answer is almost certainly negative. The reason for this is that it is known in this case that such a metric either is the standard flat metric or has conjugate points \cite{GG}, \cite{Innami}. In the case of conjugate points, we expect the heat kernel to be much larger as $t \to 0$ than the Euclidean heat kernel (see \cite{Molchanov} and \cite{KS}, although these do not cover the case of general conjugate points). 

In higher dimensions, if we assume that the norm of the Ricci curvature is an integrable function and  add the additional condition that the integral of scalar curvature vanishes (which follows automatically in the case $d=2$ from the Gauss-Bonnet theorem), then \cite{Innami} shows that again, unless the metric is flat, we have conjugate points, which should lead to a negative answer to the question. 

Supposing that the answer is negative in a given dimension $d$, it makes sense to ask
\begin{question}
Are there any Riemannian $d$-dimensional manifolds, not diffeomorphic to the standard flat $\RR^d$, for which the heat kernel is bounded above and below by multiples of the Euclidean heat kernel?
\end{question}
Here we expect that the answer is positive: examples should be furnished by asymptotically conic metrics on $\RR^d$ that have nonpositive sectional curvatures. (One can write down such metrics which take the form $dr^2 + \alpha r^2 d\omega$ for all $r \geq R$, where $d\omega$ is the standard metric on the sphere and $\alpha$ is any constant larger than $1$.) We expect that a positive answer will follow  from an analysis of the resolvent on such manifolds, parallel to the analysis in \cite{Melrose-Sa Barreto-Vasy}. 
Along the spectral axis, this has already been carried out in \cite{Guillarmou-Hassell-Sikora-SM}.


\begin{appendix}

\section{Resolvent from parametrix}\label{sec:resolvent from parametrix}

\begin{proof}[Proof of Corollary \ref{coro : resolvent construction}]In the first instance, we obtain a semiclassical resolvent $$\tilde{R}(h, \sigma) = (h^2 \Delta_X - h^2 n^2/4 - \sigma^2)^{-1} \quad \mbox{with $|\sigma| = 1$, $\Im \sigma \leq 0$ and $h \in [0, 1)$}, $$ through the parametrix $G(h, \sigma)$ constructed by Melrose, S\`{a} Barreto and Vasy. Then the properties of the resolvent $R(\lambda)$ in Corollary \ref{coro : resolvent construction} follow from the counterparts for $\tilde{R}(h, \sigma)$.

We start by recalling the result of Melrose-S\`{a} Barreto-Vasy. 

\begin{theorem}[\cite{Melrose-Sa Barreto-Vasy}]\label{thm:sclres}
Let $P(h, \sigma) = h^2\Delta_X - h^2n^2/4 - \sigma^2$ for $h \rightarrow 0$ and $|\sigma| = 1$. There exists an operator $G(h, \sigma) = G_{{diag}}(h, \sigma) + G_{od}(h, \sigma)$ such that, uniformly in $\sigma$,
$$P(h, \sigma) \circ G(h, \sigma) = Id + E(h, \sigma).$$ Here the parametrix $G = G_{{diag}} + G_{od}$ (where subscript ${diag}$ stands for `diagonal' and $od$ for `off-diagonal') and the error $E$ are pseudodifferential operators satisfying
\begin{itemize}
\item the Schwartz kernel of $G_{diag}(h, \sigma)$ at the diagonal is a semiclassical pseudodifferential operator on the $0$-cotangent bundle with semiclassical order $0$ and differential order $-2$, given by an oscillatory integral of the form
\begin{equation}
\label{eqn : semiclassical resolvent at diagonal interior}(2 \pi h)^{-n-1} \int e^{\imath ( (z-z') \cdot \zeta ) / h} a(z, z', \zeta, h, \sigma) \, d\zeta
\end{equation}
in terms of local coordinates $z$ in the interior of $X$, and
 \begin{equation}
\label{eqn : semiclassical resolvent at diagonal}(2 \pi h)^{-n-1} \int e^{\imath ( (x - x') \cdot \xi +  (y - y) \cdot \eta ) / xh} a(x, y, x', y', \xi, \eta, h, \sigma) \,d\xi d\eta
,\end{equation} in terms of local coordinates $(x,y)$ near the boundary. Here $a$ is a classical symbol of order $-2$ in the fibre variables $\zeta$, resp. $(\xi, \eta)$. We may assume that the support of $G_{{diag}}$ is in the set $\{ d(z,z') \leq 2h \}$.

\item the Schwartz kernel of $G_{od}(h, \sigma)$ is supported in the set $\{ d(z,z') \geq h \}$, and takes the form
\begin{equation}
\label{eqn : semiclassical resolvent away from diagonal}e^{- i \sigma d(z, z') / h} \rho_\mathcal{L}^{n/2} \rho_\mathcal{R}^{n/2} \rho_\mathcal{A}^{-n/2 -1} \rho_\mathcal{S}^{-n-1} \mathcal{A}^0\Big(X_0^2 \times_1 [0, 1)_h\Big),
\end{equation}
where $\mathcal{A}^0(X)$ denotes the set of $L^\infty$-based conormal functions on $X$. Notice that $$e^{- i \sigma d(z, z') / h} = e^{-i\lambda d(z,z')}$$ is exponentially decreasing for $\Im \lambda < 0$.

\item the Schwartz kernel of the error $E(h, \sigma)$ has the form
\begin{equation}\label{eqn : error mapping}
x^{\infty} h^\infty {x'}^{n/2 + i\sigma/h} \mathcal{A}^0(X^2 \times [0, 1)_h ),
\end{equation}
i.e.\ vanishes to infinite order at $x=0$,  the boundary of the left copy of $X$,  and at $h=0$.
In particular it is compact on $x^{\epsilon} L^2(X)$ for any $\epsilon > 0$.
\end{itemize}
\end{theorem}

To construct the exact resolvent from this parametrix, we need to invert the operator $\Id + + E(h, \sigma)$. We look for an operator $S(h, \sigma)$ such that $Id + S$ is the inverse of $Id + E$. In light of \eqref{eqn : error mapping}, the Schwartz kernel of $E(h, \sigma)$ on $X^2$ vanishes at the right face to $n/2$ order and at the left face to infinite order. It is more convenient to conjugate $E(h, \sigma)$ by $x$, the boundary defining function of $X$, to make the error square-integrable with respect to the Rimannian density, which is a smooth multiple of $x^{- (n + 1)} dxdy$. Let us consider $E^c = x^{-1} E x$ instead. Then the kernel of $E^c$ is an $L^2$-integrable function on $X^2$, that is, $E^c$ is a Hilbert-Schmidt operator. Moreover, it vanishes at the semiclassical face to infinite order, which implies $\|E^c\|_{L^2(X^2)} = O(h^\infty)$. Consequently, $Id + E^c$ is invertible for small enough $h$, say $h \leq h_0$, and the inverse is denoted by $Id + S^c$. Notice that $S^c$, like $E^c$, is Hilbert-Schmidt with Hilbert-Schmidt norm vanishing to infinite order as $h \to 0$. One can observe that the desired operator $S$ is indeed $x S^c x^{-1}$.

Secondly, we claim $S$ obeys \eqref{eqn : error mapping}. To see this, we use the fact that $$Id = (Id + E^c)(Id + S^c) \quad \mbox{and} \quad Id = (Id + S^c)(Id + E^c).$$ These two identities yield $$S^c = - E^c + E^c E^c + E^c S^c E^c.$$ It is obvious that $E^c$ and $E^c E^c$ lie in the space
\begin{equation}\label{eqn : conjugated error mapping}
x^{\infty} h^\infty {x'}^{n/2 + 1+ i\sigma/h} \mathcal{A}^0(X^2 \times [0, h_0)_h ),
\end{equation}
 On the other hand, we write the kernel of $E^c S^c E^c$ as $$E^c S^c E^c (z, z') = h^\infty \int_X \int_X x^\infty A(z, z'',h)S^c(z'', z''',h)B(z''', z',h)(x')^{n/2 + 1 + \imath \sigma/h} dg(z'')dg(z'''),$$ where $A$ is a conormal function of $z$ square-integrable in $z''$ and $B$ is a conormal function in $z'$ and square-integrable in $z'''$, both uniformly in $h$. Since $S^c$ is also in $L^2(X^2)$, uniformly in $h$, it follows that $E^cS^cE^c$ also obeys \eqref{eqn : conjugated error mapping}. Then we have proved the claim.

The next step is to get the resolvent. Write $$\tilde{R}(h, \sigma) = G(h, \sigma) (Id + S(h, \sigma)).$$ We assert \begin{lemma}\label{lemma : resolvent od} $\tilde{R}(h, \sigma)$ can written as $$\tilde{R}(h, \sigma) = \tilde{R}_{{diag}}(h, \sigma) + \tilde{R}_{od}(h, \sigma),$$ such that $\tilde{R}_{{diag}}(h, \sigma)$ obeys \eqref{eqn : semiclassical resolvent at diagonal} and \eqref{eqn : semiclassical resolvent at diagonal interior}, whilst $\tilde{R}_{od}(h, \sigma)$ takes the form of \eqref{eqn : semiclassical resolvent away from diagonal}.\end{lemma}
Assuming the lemma for the moment, we have so far obtained the desired resolvent $\tilde{R}(h, \sigma)$ of the same structure with $G(h, \sigma)$.
Noting that $\lambda^2 R_{ac}(\lambda) = \tilde{R}(h, \sigma)$ with $h = \lambda^{-1}$, one can easily deduce \eqref{eqn : resolvent decomposition}, \eqref{resolvent away from diag}, \eqref{Rod large r} and \eqref{Rod small r}, which proves the corollary.

It remains to prove Lemma \ref{lemma : resolvent od}.
 In fact, Theorem \ref{thm:sclres} yields that 
 $$
 \tilde{R}(h, \sigma) =  G_{diag}(h, \sigma) + G_{od}(h, \sigma) + G_{diag}(h, \sigma) S(h, \sigma) + G_{od}(h, \sigma) S(h, \sigma).
 $$ 
 It is clear that $G_{diag}(h, \sigma)$ and $G_{od}(h, \sigma)$ satisfy the conditions of the theorem. Next consider $G_{diag}(h, \sigma) S(h, \sigma)$. For a manifold with boundary, $M$, we use $\CIdot(M)$ to denote the space of $C^\infty$ functions on $M$ such that all derivatives at the boundary vanish. 
 We may view $S(h, \sigma)$ as a kernel in ${x'}^{n/2 + i\lambda} \CIdot([0, h_0); \mathcal{A}_0(X; C^\infty(X)))$. As a semiclassical 0-pseudodifferential operator, $G_{diag}(h, \sigma)$ maps $L^2(X)$ to $L^2(X)$ with a bound uniform in $h$. Moreover, if we compose with a finite number of semiclassical 0-derivatives, that is, a sum of products of at most $m$ derivatives of the form $h V$ where $V$ is a 0-vector field (e.g. $V = x \partial_x$ or $V = x \partial_{y_i}$ near the boundary), then this will map
 ${}^0H^m_{scl}(X) \to L^2(X)$ with a uniform bound. Here ${}^0H^m_{scl}(X)$ is the Sobolev space with squared norm given by 
 $$
 \sum_{k=0}^m \sum_{i_1, \dots, i_k = 1}^N \| h^k V_{i_1} \dots V_{i_k} u \|_{L^2}^2,
 $$
 for $V_1, \dots, V_N$ a family of vector fields that span the 0-tangent space at each point of $X$. Using this fact, and Sobolev embedding, it is straightforward to show that $G_{diag}(h, \sigma)$ maps the space $\CIdot([0, h_0); \CIdot (X))$ to itself. It follows that $G_{diag}(h, \sigma) S(h, \sigma)$ is a kernel of the form ${x'}^{n/2 + i\lambda} \CIdot([0, h_0); \CIdot(X; \mathcal{A}_0(X)))$, which is contained in the space $(xx')^{n/2 + i\lambda} h^\infty \mathcal{A}_0(X^2_0\times [0, h_0))$. 
 
Next, we consider $G_{od}(h, \sigma)S(h, \sigma)$. 
One can think of the kernel of $G_{od}(h, \sigma)$ as a function on $X^2_0 \times_1 [0, h_0)$, the kernel of $S(h, \sigma)$ as a function on $X^2 \times [0, h_0)$, and  the product of the kernels of $G_{od}(h, \sigma)$ and $S(h, \sigma)$ as a function on $X^2_0 \times_1 [0, h_0) \times X$. Both of the first two spaces can be obtained by natural projections with blowdown maps from the last one. We denote \begin{eqnarray*}\beta_G : X^2_0 \times_1 [0, h_0) \times X &\longrightarrow & X^2_0 \times_1 [0, h_0)\\ \beta_S : X^2_0 \times_1 [0, h_0) \times X &\longrightarrow &X^2 \times [0, h_0).  \end{eqnarray*} Then we have
 \begin{eqnarray*}
 \beta_G^\ast  G_{od}(h, \sigma) &\in& \rho_{\mathcal{L}}^{n/2 + \imath  \sigma/h} \rho_{\mathcal{R}}^{n/2 + i \sigma/h} \rho_{\mathcal{S}}^{- n - 1} \rho_{\mathcal{A}}^{- n/2 - 1}  \mathcal{A}^0(X^2_0 \times_1 [0, h_0) \times X) \\ \beta_S^\ast  S(h, \sigma) &\in& \rho_{\mathcal{R}}^{\infty} \rho_{\mathcal{F}}^{\infty} h^\infty {x'}^{n/2 + \imath \sigma/h}  \mathcal{A}^0(X^2_0 \times_1 [0, h_0) \times X),
 \end{eqnarray*}
 where $x'$ is the boundary defining function of $X$ and $\rho_{\mathcal{L}}, \rho_{\mathcal{R}}, \rho_{\mathcal{S}}, \rho_{\mathcal{A}}$ are the boundary defining functions of $X^2_0 \times_1 [0, h_0)$. Due to the rapid vanishing properties of $S$, the product of the two kernels on $X^2_0 \times_1 [0, h_0) \times X$ vanishes rapidly at all boundary hypersurfaces except at $\rho_{\mathcal{L}} = 0$ and $x' = 0$.

 We can then realize the kernel of the composition $G_{od}(h, \sigma) \circ S(h, \sigma)$ as the \emph{pushforward} of the product of the two distributions on $X^2_0 \times_1 [0, h_0) \times X$, via the map that first blows down to $X^3 \times [0, h_0)$ and then projects off the middle factor of $X$ (corresponding to integrating out the ``inner variable''). Let us call this map $\beta_R : X^2_0 \times_1 [0, h_0) \times X \longrightarrow X^2 \times [0, h_0)$.

 The map $\beta_R$ is a b-fibration in the sense of \cite{Melrose-conormal}, so we can apply the  pushforward theorem from \cite[Theorem 5]{Melrose-conormal} to it. This theorem shows that $$(\beta_R)_{\ast} \big(G_{od}(h, \sigma) S(h, \sigma)\big) \in x^{n/2 + i \sigma/h} (x')^{n/2 + \imath  \sigma/h} h^\infty  \mathcal{A}^0(X^2 \times [0, h_0)).$$ Combined with the result for $G_{diag}(h, \sigma) S(h, \sigma)$, we see that $G(h, \sigma) S(h, \sigma)$ is in the space $$x^{n/2 + i \sigma/h} (x')^{n/2 + \imath  \sigma/h} h^\infty  \mathcal{A}^0(X^2 \times [0, h_0)).$$
 Using Proposition~\ref{prop:dist}, with $r$ denoting geodesic distance, this can be written $$e^{(-n/2+ i\lambda)r} h^\infty\mathcal{A}^0(X^2 \times [0, h_0)),$$ since $xx'= (\rho_L \rho_R) \rho_F^2$. This completes the proof. 
 \end{proof}

\end{appendix}

\begin{flushleft}
\vspace{0.3cm}\textsc{Xi Chen\\Shanghai Center for Mathematical Sciences\\Fudan University, Shanghai 200433, China}

\emph{E-mail address}: \textsf{xi\_chen@fudan.edu.cn}

\end{flushleft}

\begin{flushleft}
\vspace{0.3cm}\textsc{Andrew Hassell\\Mathematical Sciences
Institute\\Australian National University, Canberra 2601, Australia}

\emph{E-mail address}: \textsf{andrew.hassell@anu.edu.au}

\end{flushleft}

\end{document}